\documentclass[reqno]{amsart}						% for AMS article class
\usepackage{microtype}
\usepackage{fullpage}

\usepackage{amsmath}						% for AMS macros
\usepackage{amssymb}						% for AMS symbols
\usepackage{amsfonts}						% for AMS fonts
\usepackage{amsthm}	
\usepackage[foot]{amsaddr}						% for author footnotes
\usepackage{array}

\usepackage{mathtools}						% for advanced math
\mathtoolsset{%
}

\usepackage{amsopn}
\numberwithin{equation}{section}

\usepackage[utf8]{inputenc}							% for source encoding
\usepackage[T1]{fontenc}						% for font encoding

\usepackage[proportional,lining]{libertine}
\usepackage[libertine,libaltvw,cmintegrals,varbb]{newtxmath}

\usepackage{dsfont}							% for blackboard bold font
\usepackage[%								% for math font selection
cal=cm,
]
{mathalfa}

\usepackage[labelfont={bf,small},labelsep=colon,font=small]{caption}	% for caption control
\usepackage[dvipsnames,svgnames]{xcolor}						% for color
\colorlet{MyBlue}{DodgerBlue!75!Black}
\colorlet{MyGreen}{DarkGreen!85!Black}

\renewcommand{\paragraph}{\subsection}
\usepackage{subcaption}
\usepackage{acronym}						% for acronyms
\usepackage{booktabs}						% for tables
\usepackage{latexsym}						% for symbols
\usepackage{paralist}						% for inline lists
\usepackage{xspace}							% for flexible space

\usepackage{hyperref}
\hypersetup{
final,
colorlinks=true,
linktocpage=true,
pdfstartview=FitH,
breaklinks=true,
pdfpagemode=UseNone,
pageanchor=true,
pdfpagemode=UseOutlines,
plainpages=false,
bookmarksnumbered,
bookmarksopen=false,
bookmarksopenlevel=1,
hypertexnames=true,
pdfhighlight=/O,
urlcolor=Maroon,linkcolor=MyBlue!60!black,citecolor=DarkGreen!70!black,	% for on-screen
pdftitle={},
pdfauthor={},
pdfsubject={},
pdfkeywords={},
pdfcreator={pdfLaTeX},
pdfproducer={LaTeX with hyperref}
}

\theoremstyle{plain}
\newtheorem{theorem}{Theorem}[section]						% for theorems
\newtheorem*{theorem*}{Theorem}						% for theorems
\newtheorem*{corollary*}{Corollary}						% for corollaries (unnumbered)
\newtheorem{lemma}{Lemma}[section]					% for lemmas
\theoremstyle{definition}
\newtheorem*{definition*}{Definition}					% for definitions (unnumbered)

\theoremstyle{remark}
\newtheorem{remark}{Remark}						% for remarks
\newtheorem*{remark*}{Remark}						% for remarks (unnumbered)
\newtheorem*{example*}{Example}						% for examples (unnumbered)

\usepackage{tikz}							% for figures
\usepackage{pgfplots}
\usetikzlibrary{calc,patterns,matrix}
\usepackage{algorithmic}
\usepackage[section]{algorithm}

\usepackage{xspace}							% for flexible space

\DeclareMathOperator*{\argmin}{argmin}



\newcommand{\supp}{\mathrm{supp}}

\newcommand{\RR}{\mathbb{R}}
\newcommand{\EE}{\mathbb{E}}

\newcommand{\nullC}{\mathrm{null}}

\newcommand{\barx}{\bar{x}}
\newcommand{\pmin}{p_{\min}}

\newcommand{\pmax}{p_{\max}}

\newcommand{\prox}{\mathbf{prox}}
\newcommand{\rgprox}{\mathbf{prox}_{\gamma r}}
\newcommand{\rprox}{\mathbf{prox}_{r}}

\newcommand{\pp}{\pi}
\newcommand{\pvec}{{p}}

\newcommand{\dave}{\texttt{DAve-PG}\xspace}
\newcommand{\salgo}{\texttt{Spy}\xspace}
\newcommand{\recoalgo}{Reconditioned--\texttt{Spy}\xspace}
\newcommand{\cataalgo}{Catalyst--\texttt{Spy}\xspace}

\usepackage[algo2e]{algorithm2e} 
\usepackage{algorithmic}
\usepackage{algorithm}

\usepackage{graphicx}
\usepackage{subcaption}
\usepackage{tcolorbox}
\usepackage{xcolor}
\usepackage{tikz}

\patchcmd{\SetTagPlusEndMark}{$}{}{}{}
\patchcmd{\SetTagPlusEndMark}{$}{}{}{}

\newtheorem{assumption}{Assumption}

\usepackage{enumitem}
\setlist[enumerate]{leftmargin=.5in}
\setlist[itemize]{leftmargin=.5in}

\title{Distributed Learning with Sparse Communications by Identification}
\author[D.~Grishchenko]{Dmitry Grishchenko$^{\ast}$}
\email{dmitry.grishchenko@univ-grenoble-alpes.fr}

\author[F.~Iutzeler]{Franck Iutzeler$^{\ast}$}
\email{franck.iutzeler@univ-grenoble-alpes.fr}

\author[J.~Malick]{\\Jérôme Malick$^{\diamond}$}
\email{jerome.malick@univ-grenoble-alpes.fr}

\author[M.~Amini]{Massih-Reza Amini$^{\dagger}$}
\email{massih-reza.amini@univ-grenoble-alpes.fr}

\address{$^{\ast}$  Univ. Grenoble Alpes, LJK, 38000 Grenoble, France}
\address{$^{\diamond}$ CNRS, LJK, 38000 Grenoble, France}
\address{$^{\dagger}$  Univ. Grenoble Alpes, LIG, 38000 Grenoble, France}


\begin{document}

\maketitle

% REQUIRED
\begin{abstract}
 In distributed optimization for large-scale learning, a major performance limitation comes from the communications between the different entities. When computations are performed by workers on local data while a coordinator machine coordinates their updates to minimize a global loss, we present an asynchronous optimization algorithm that efficiently reduces the communications between the coordinator and workers. This reduction comes from a random sparsification of the local updates. We show that this algorithm converges linearly in the strongly convex case and also identifies optimal strongly sparse solutions. We further exploit this identification to propose an automatic dimension reduction, aptly sparsifying all exchanges between coordinator and workers.
%This "sparsification by identification" is complementary to other sparsification techniques.
\end{abstract}

% ===========================================================
% ===========================================================
\section{Introduction}
% ===========================================================
% ===========================================================

% ===========================================================
\subsection{Context and contributions}
% ===========================================================

We consider a distributed learning setup where observations are scattered over many machines. We assume that the machines, also referred to as \emph{workers}, have a private subset of the examples, perform their computations independently, and communicate with a \emph{coordinator} machine. 
Standard learning approaches would consider that the entire training set is moved and stored in one single machine or in a datacenter. In contrast, we consider the approach consisting in learning collaboratively a shared prediction model without moving the training data, as for instance in federated learning (see e.g.\;the recent review articles\;\cite{kairouz2019advances} and\;\cite{li2019federated}).
In this context, communications are typically the practical bottleneck of the learning process; see e.g.\,\cite{li2013distributed,konevcny2016federated,lin2017deep,wangni2018gradient,ma2019accelerated}.

Recently, several approaches to mitigate this communication issue have been proposed, including different forms of gradient or model compression, which reduces representation size by quantization (reduce precision on elements) or sparsification (send only most significant gradient entries); we refer e.g.\;to early works \cite{seide20141, strom2015scalable} and more recent developments
\cite{khirirat2018gradient,stich2018sparsified,alistarh2018convergence,basu2019qsparse,vogels2019powersgd} among many others.

In this paper, we address the question of reducing the size of communications with an approach complementary to existing compression techniques. We propose  to adaptively sparsify the model updates using its progressively uncovered sparse structure. More precisely, we first propose an asynchronous distributed algorithm featuring a sparsification of upwards communications (worker-to-coordinator) by randomly zeroing the local update entries. This randomized technique is adjustable to various levels of communication costs, machines' computational powers, and data distribution evenness. Also, an attractive and original property of this algorithm is the possibility to use a fixed learning rate that does not depend neither on communication delays nor on the number of machines. However, this random sparsification technique provably works only for i.i.d. sparsifications with either almost-uniform distributions, or well-conditioned problems. This makes aggressive sparsification or adaptation to the sparsity structure of the model impossible with such an algorithm.

To tackle this issue, we use a proximal reconditioning scheme wrapping up the previously mentioned algorithm as an inner minimization method. This allows us to perform much more aggressive sparsifications. Furthermore, we show that when using a sparsity-inducing regularizer, our reconditioned algorithm generates iterates that identify the optimal sparsity pattern of the model in finite time. This progressively uncovered pattern can be used to adaptively update the sparsification distribution of the inner method. All in all, this method only features sparse communications: the downward communications (coordinator-to-worker) consists in sending the (eventually) sparse model, and the the upwards communications (worker-to-coordinator) are adaptively and aggressively sparsified.

Finally, we show theoretically and numerically that our method has better performance than its non-sparsified version, in terms of suboptimality with respect to the quantity of information exchanged.

% ============================================================
\subsection{Related literature on sparsification}
\label{sec:existing}
% ============================================================

A popular way of sparsifying models is to use proximal methods: they consist in using a sparsity-inducing penalty $r$ (typically $\ell_1$ or $\ell_{1,2}$ norms, see e.g.\;\cite{bach2012optimization}), and handling it with a proximity operator defined as the $\RR^d\rightarrow\RR^d$ mapping 
\begin{equation}\label{eq:prox}
    \rprox(x) = \argmin_{y\in\RR^d}\left\{r(y) + \frac{1}{2} \|y-x\|^2\right\} .
\end{equation}
For many popular regularizations, this operator has an explicit formulation or can be efficiently computed. For example, the proximity operator of the $\ell_1$-norm (with parameter $\lambda$) is given coordinate-wise\footnote{We denote by $x_{[j]} \in \mathbb{R}$ the $j$-th coordinate of vector $x$.} as
\begin{align}\label{eq:soft}
  \left(  \prox_{ \lambda \|\cdot\|_1 }(u) \right)_{[j]} = \left\{
    \begin{array}{ll}
    ~ 0   & \text{ if } u_{[j]} \in [ - \lambda ,  \lambda ]  \\
    ~ u_{[j]}  - \lambda & \text{ if } u_{[j]} > \lambda \\
    ~ u_{[j]}  + \lambda & \text{ if } u_{[j]}  < -\lambda \\
    \end{array}
    \right.\qquad \text{for all $j$},
\end{align} 
    which has a clear sparsification effect. This kind of proximal sparsification was successfully used in machine learning (see e.g.\;the seminal work\;\cite{donoho1995noising}), for algorithmic 
    algorithmic perspectives (see e.g.\;\cite{langford2009sparse}) and in the case of synchronous distributed systems (see e.g. \cite{wang2017efficient,smith2015l1}).

Block coordinate descent algorithms, which knew a rebirth in the context of huge-scale learning~\cite{nesterov2012efficiency,richtarik2016distributed}, can also be interpreted as specific sparsification of gradient updates. In the context of $\ell_1$-regularized problems, they can be combined with screening techniques for efficiently solving high-dimensional problems\;\cite{massias2018dual}. We also mention that the idea of randomly selecting some entries has been already applied in distributed algorithms: random selection is used to sparsify local gradients in the synchronous algorithm of \cite{wangni2018gradient}, to sparsify the variance-reducing term in the stochastic algorithm of \cite{pedregosa2017breaking}, or to sparsify updates in fixed-point iterations~\cite{peng2016arock}. 
% {\color{red}  In \cite{mishchenko2019distributed} authors propose distributed gradient-type method with compression of gradient differences. The key idea is in quantization of difference of the current stochastic gradient on each worker with an error term stored locally on each worker machine.}    
% In the context of federated learning, \cite{konevcny2016federated} mentions the idea of random sparsification, but without further study. Here we provide a thorough analysis of the idea, and we further study its interplay with the proximal approach and with identification properties.

Our proposed algorithm combines these three approaches (proximal methods, coordinate descent, and random selection): the output of a proximal step is sent downwards (coordinator-to-workers) and random coordinate updates are sent upwards (workers-to-coordinator); see details in Section~\ref{sec:sparse}. Beyond the original combination of known ingredients, the specificity of our work is the study of the adaptive sparsification and its interplay with identification properties.

Identification of optimal structure has been extensively studied in the context of constrained convex optimization (e.g.\;\cite{wright1993identifiable}) and nonsmooth nonconvex optimization (e.g.\;\cite{hare2004identifying}). 
Although many deterministic algorithms have been proved to have identification properties, the situation is less clear for randomized algorithms: in particular (proximal) stochastic gradient methods are known to be unable to identify substructure (see e.g.\,\cite{lee2012manifold}), while it has been recently proved that (proximal) variance-reduced stochastic gradient methods do have such identification properties \cite{poon2018local,fadili2018model}. No identification results have yet been reported for asynchronous distributed optimization algorithms.

For our algorithm, we use identification to show that the iterates produced by the coordinator become sparse when the regularization $r$ is a sparsity-inducing norm \cite{bach2012optimization} and also leverage the sparsity pattern identified by the coordinator to adapt the sparsification at the workers. Finally note that, though identification has already been used to accelerate algorithms (by e.g.\;a better tuning of parameters \cite{liang2014local}), it is the first time that it is used for an automatic exchange reduction and sparsification.

% %=============================================================
% %=============================================================
% \subsection{Organization of the paper}
% %=============================================================
% %=============================================================

% We first introduce and analyze an extension of Algorithm\;\ref{alg:dave}
% %\dave \cite{ICML18} 
% featuring sparsified iterations. We show that (i) although this sparsification indeed leads to sparse iterates, the performances of the algorithm are deteriorated theoretically and practically; (ii) using a performing adaptive sparsification strategy implies having to recondition the problem.   

% ============================================================
% ============================================================
%\section{Sparsification \& Asynchronous Distributed Optimization}
% ============================================================
% ============================================================
% =================================================================
% =================================================================
\section{Preliminaries: problem, notations, and recalls}\label{sec:prelim}
% =================================================================
% =================================================================

% ============================================================
%\vspace*{1ex}
%\paragraph{Distributed setup} 

We consider a distributed learning setup where $n$ observations are scattered over $M$ machines; each machine $i$ having a private subset $\mathcal{S}_i$ of the examples. This %Learning from such distributed data 
leads to optimization problems of the form
\begin{align}\label{eq:pb}
\tag{$\mathsf{P}$}
\min_{x\in\mathbb{R}^d}  ~~ F(x) =  \sum_{i=1}^M  \alpha_i  f_i(x)  +  r(x),
\end{align}
with $f_i(x) = \frac{1}{|\mathcal{S}_i|}  \sum_{j\in\mathcal{S}_i}l_j(x) $ the local empirical risk at machine\;$i$ ($l_j$ standing for the smooth loss function for example $j$). The $\alpha_i$ are the proportion of observations locally stored in machine\;$i$, hence:
%$\sum_{i=1}^M\alpha_i = 1$.
$$
\alpha_i= |\mathcal{S}_i|/n \qquad\text{and}\qquad \sum_{i=1}^M\alpha_i = 1.
$$

The function $r$ is a regularization term introducing a {prior} on the {structure} of the model\;$x$, improving both its interpretability and its stability. One of the first and most well-known instances of such a prior is the sparsity sought in the parameters of least-squares regression with $\ell_1$-regularization $r=\|{\cdot}\|_1$), leading to the well-known lasso problem \cite{tibshirani1996regression}. The notation related to sparsity is the following: we denote the support of a vector $x\in\mathbb{R}^d$ by  $\supp(x) =  \{ i\in \{1,\dots,d\} : x_{[i]} \neq  0\}$, and its size by $|\supp(x)|$. Oppositely, we define the sparsity pattern of a vector $x\in\mathbb{R}^d$ by the set $\nullC(x) = \{ i\in \{1,\dots,d\} : x_{[i]} =  0\}$, and denote its size by $|\nullC(x)|$. Note that these two index sets are complementary; $\supp(x) = \overline{\nullC(x)}$.

We consider the problem \eqref{eq:pb} in a standard convex setting, with the function $r$ convex and lower-semicontinuous (lsc), and all the functions $(f_i)$ $L$-smooth and $\mu$-strongly convex (with $\mu\geq 0$). 
%The fact that the $f_i$'s have the same constants is not limiting, as this achieved easily by exchanging quadratic terms between functions. 
We then define the \emph{condition number} of the smooth part of~\eqref{eq:pb} as 
$\kappa_{\eqref{eq:pb}} = \mu/L.$
%$$\kappa_{\eqref{eq:pb}} = \frac{\mu}{L}.$$

% ============================================================
\vspace*{1ex}
\paragraph*{On asynchronous distributed optimization} 
The general computation framework is the following: one machine gets the current model from the coordinator, improves it from its own data which produces a local model update; only this update is sent to the coordinator, which computes a new shared model.
There exist many methods for solving problems of this type by iteratively sending small messages or model updates during the training process (as opposed to sending raw data); see e.g.\;\cite{zhang2014asynchronous,aytekin2016analysis,peng2016arock}. The literature on distributed optimization methods without shared memory usually use on restrictive assumptions on the computing system delays, which in turn impact the obtained convergence rates. Asynchronous coordinate descent methods which are able to handle unbounded delays were recently proposed in \cite{hannah2016unbounded,sun2017asynchronous} but use decreasing stepsizes. In contrast, the asynchronous proximal-gradient algorithm of \cite{ICML18,mishchenko2018} allows the use of fixed stepsizes (the same as in the vanilla proximal gradient algorithm), and their delay-independent analysis technique result an linear convergence in the strongly convex case.
We present below the notation and the asynchronous algorithm of \cite{ICML18}, which is going to be the baseline for our developments. 

%=============================================================
\vspace*{1ex}
\paragraph*{Notation: time, delays, and epochs}

As asynchronous distributed setup allows the algorithm to carry on computations without waiting for slower machines: the machines work on outdated versions of the main variable, and the coordinator gathers the workers' inputs on the fly to produce the updates. The following notation formalizes this framework:
\begin{itemize}
\item \emph{For the coordinator.} We define the time\;$k$, as the number of updates the coordinator has received from any of the workers. At time $k$, the coordinator receives some input from an agent, denoted by $i^k$, updates its global variable, and sends it to worker\;$i^k$.
\item \emph{For worker $i$.} At time $k$, we introduce $d_i^k$ as the time elapsed since the last time the coordinator received an update from worker $i$. In particular, we have $d_{i_k}^k = 0$.
We also consider $D_i^k$ as the elapsed time since the penultimate update. This means that, at time $k$, the last two moments when the worker $i$ updates the coordinator are $k-d_i^k$ and $k-D_i^k$. .
\end{itemize}

\smallskip
\noindent We consider a \emph{totally asynchronous} as per the classification of the book\;\cite[Chap.\;6.1]{bertsekas1989parallel}, meaning that all workers are responsive. However the worker responses can have arbitrary delays.
Following \cite{mishchenko2018} we introduce the notion of epochs, as the sequence of stopping times $(k_m)$ defined as $k_0 = 0$ and
\begin{align}
\label{eq:km}
k_{m+1} &= \min\left\{ k : k- D_i^k \geq k_m ~\text{for all $i$} \right\}.
\end{align}
This sequence fully captures the asynchronicity and allows us to analyze algorithms independently of the computing system. All our convergence results involve this notion of epochs.

%=============================================================
\vspace*{1ex}
\paragraph*{Asynchronous proximal-gradient and its rate for strongly convex objectives}

As explained above, our baseline is the asynchronous proximal-gradient algorithm of \cite{ICML18}, called \dave, which solves distributed problems of the form~\eqref{eq:pb} represented by the triplet $((\alpha_i),(f_i),r)$ for (workers weights, workers functions, global regularization); see Algorithm~\ref{alg:dave}. 
This algorithm converges linearly at a rate that only depends on the function properties but neither on the number of machines nor on the delays, %as they are 
directly embedded in the sequence $(k_m)$.

\begin{algorithm}
\caption{\textsc{\dave} on $((\alpha_i),(f_i),r)$ with stopping criterion $\mathsf{C}$} 
\label{alg:dave}

\tcbset{width=0.48\columnwidth,before=,after=\hfill, colframe=black,colback=white, fonttitle=\bfseries, coltitle=white, colbacktitle=black, boxrule=0.2mm,arc=0mm, left = 2pt}

\begin{tcolorbox}[title=\textsc{Coordinator} \vphantom{\texttt{Worker  $i$}}]
Initialize $\bar x^0$\\
\While{%stopping criterion
 $\mathsf{C}$ is not verified}{
    \hspace*{0ex}\\
     {\color{blue!70!black}Receive\,$\Delta^k$\,from\,agent\,$i^k$}\\
     $\barx^k \leftarrow \barx^{k-1} + \alpha_i \Delta^k $\\
     $ x^k \leftarrow \rgprox(\bar x^k )$\\
    {\color{red!80!yellow} Send $x^k$ to agent $i^k$}\\
    $k\leftarrow k+1$
}
%Interrupt all workers\\ % XX commenté pour simplifier
%\textbf{Output} $x^k$ 
\end{tcolorbox}
\tcbset{width=0.48\columnwidth, colframe=black!50!black, coltitle=white, colbacktitle=black}
\begin{tcolorbox}[title=\textsc{Worker } $i$]
Initialize $x_i = x_i^+ = 0$\\
\While{not interrupted}{
    \hspace*{0.1ex}\\
     {\color{red!80!yellow}Receive  $x$ from the coordinator}\\
    $x^+_i  \leftarrow x - \gamma \nabla f_i( x)$\\%[0.5ex]
    $\Delta \leftarrow x_i^+ - x_i $ \\%[0.5ex]
    {\color{blue!70!black}Send $\Delta$ to the coordinator}\\%[0.5ex]
    $x_i \leftarrow x^+_i $%\\[0.7ex]
}
\vspace*{-0.05cm}
\end{tcolorbox}
\vspace*{-0cm}
\end{algorithm}

\begin{theorem}[Th.\;3.2 of\;\cite{mishchenko2018}]\label{th:davepg}
Let the functions $(f_i)$ be $\mu$-strongly convex ($\mu>0$) and $L$-smooth. Let $r$ be convex lsc. Using $\gamma \in (0, \frac{2}{\mu + L}]$, \dave converges linearly on the epoch sequence $(k_m)$. More precisely, for all $k\in [k_m, k_{m+1})$
\begin{align*}
    \left\| x^k - x^\star \right\|^2 \le
    \left\| \bar x^k - \bar x^\star \right\|^2 \le
    \left(1 - \gamma \mu  \right)^{2m} \max_i\left\|x_i^0-x_i^\star\right\|^2,
\end{align*}
with the shifted local solutions $x_i^\star = x^\star - \gamma_i\nabla f_i(x^\star)$ and  $\bar x^\star = \sum\alpha_i x_i^\star$.

Furthermore, using the maximal stepsize $\gamma = \frac{2}{\mu + L}$, we obtain for all $k\in [k_m, k_{m+1})$
\begin{align*}
    \left\| x^k - x^\star \right\|^2 \le \left( \frac{1-\kappa_{\eqref{eq:pb}}}{1+\kappa_{\eqref{eq:pb}}}  \right)^{2m} \max_i\left\|x_i^0-x_i^\star\right\|^2.
\end{align*}
\end{theorem}

In this result, the stepsize $\gamma$ can be taken in the usual range for (proximal) gradient descent (in contrast with existing asynchronous algorithms, even for the works with the most realistic assumptions; see\;e.g.\;\cite{aytekin2016analysis,liu2015asynchronous,sun2017asynchronous} and references therein).

\section{Sparsification of upward communications}\label{sec:sparse}

In this section, we present a first distributed algorithm for solving \eqref{eq:pb} with sparse workers-to-coordinator communications. 
% It is based on sparsifying the local updates of Algorithm\;\ref{alg:dave}. We emphasize the relation between sparsification strategies and convergence issues.

%%%%%%%%%%%%%%%%%%%%%%%%%%%%%%%%%%%%%%%%%%%%%%%%%%
\subsection{Random sparsification of local updates}\label{sec:algo}
%%%%%%%%%%%%%%%%%%%%%%%%%%%%%%%%%%%%%%%%%%%%%%%%%%%

In the proposed method, the coordinator machine asynchronously gathers \emph{sparsified} delayed updates from workers and sends back the global variable. More specifically, each worker independently compute a gradient step on its local loss for a randomly drawn subset of coordinates only.
The coordinator machine keeps track of the weighted average of the most recent worker outputs, computes the proximity operator of the regularizer at this average point, and sends this result back to the updating worker\;$i^k$. 
% J commented out
% Maintaining a weighted average of worker outputs is a special feature, inspired from \dave: though it may appear conservative, it actually performs well in practice due to the stability of the produced iterations. The intuitive reason is that combining delayed points is more stable than using a combination of delayed directions; see the numerical comparisons of Section 2.4 of \cite{mishchenko2018} and Section 5 of \cite{ICML18}. 
At iteration $k$, the random subset of entries that worker $i^k$ updates is denoted by $\mathbf{S}^k$ (in bold, emphasizing that it is the only random variable in the algorithm). The update~writes 
\begin{subequations}%\label{alg:base}
\begin{align*}%\label{alg:base1}
x_{i[j]}^k &= \left\{ 
\begin{array}{ll}
 \!\!\!\left( x^{k-D_i^k}    - \gamma \nabla f_i( x^{k-D_i^k}) \right)_{[j]}  \hspace*{-0.2cm}  &  \hspace*{0ex} \text{ if } \left| \begin{array}{l}
    i = i^k\\ j\in\mathbf{S}^{k-D_i^k}
 \end{array} \right.  \\[0.5cm]
x_{i[j]}^{k-1} & \hspace*{1ex}\text{ otherwise }
\end{array}
\right. \\[-0.5ex]
x^k &= \rgprox \left(\barx^k\right) \qquad \text{with }\quad \barx^k=\sum_{i=1}^M \alpha_i x_i^k ,
\end{align*}
\end{subequations}
With this sparsification, the local updates correspond to a random block coordinate descent step for the workers. However, this algorithm does not boil down to an asynchronous stochastic block-coordinate descent algorithm such as \cite{liu2015asynchronous,sun2017asynchronous,peng2016arock,richtarik2016distributed}, since our method maintains a variable, $\barx^k$, aggregating asynchronously all the workers' contributions.

\begin{assumption}[On the random sparsification]
\label{hyp:algo}
The sparsity selectors $(\mathbf{S}^k)$ are independent and identically distributed random variables. We select a coordinate in
$\mathbf{S}^k$ as follows:
\[
\mathbb{P}[j\in\mathbf{S}^k] = p_j > 0 ~~~~~ \text{for all $j\in\{1,\dots,d\}$,}
\]
with $\pvec = (p_1,\dots,p_d) \in(0,1]^d$.
We denote $\pmax = \max_i p_i ~\text{and}~ \pmin = \min_i p_i.$
\end{assumption}

The selectors $(\mathbf{S}^k)$ being the only random variables of the algorithm,  it is natural to define the filtration $\mathcal{F}^k = \sigma( \{\mathbf{S}^\ell\}_{\ell<k} )$ so that all variables at time $k$ ($x_i^k$, $\barx^k$, $x^k$, $d_i^k$, $D_i^k$) are $\mathcal{F}^k$-measurable but $\mathbf{S}^k$ is not. %This filtration will be used in the proofs.

%%%%%%%%%%%%%%%%%%%%%%%%%%%%%%%%%%%%%%%%%%%%%%%%%%%
\subsection{Distributed implementation}
%%%%%%%%%%%%%%%%%%%%%%%%%%%%%%%%%%%%%%%%%%%%%%%%%%%

The algorithm introduced in the previous section can be naturally distributed. We formalize this in Algorithm\;\ref{alg:spy}, that we call \salgo. Similarly to \dave, this method is generic in the sense that none of its ingredients, including the stepsize choice, depend on the computing system. It also shares the feature that although each coordinator update involves only one agent (and thus part of the data), all the data is always implicitly involved in the coordinator variable; which allows the algorithm to cope with the heterogeneity of the computing system (data distribution, agents delays). Its presentation uses the following notation: for a vector of $x\in\mathbb{R}^d$ and a subset $S$ of $\{1,\dots,d\}$, $[x]_{S}$ denotes the sparse size-$d$ vector where $S$ is the set of non-null entries, for which they match those of\;$x$, i.e.\;$([x]_{S})_{[i]} = x_{[i]}$ if $i\in S$ and $0$ otherwise.

% The algorithm \salgo has the same arguments as \dave plus a probability vector (see Assumption~\ref{hyp:algo}).

\begin{algorithm}
\caption{\textsc{\salgo} on $((\alpha_i),(f_i), r  ~ ; ~  p)$ with stopping criterion $\mathsf{C}$}
\label{alg:spy}

\tcbset{width=0.48\columnwidth,before=,after=\hfill, colframe=black,colback=white, fonttitle=\bfseries, coltitle=white, colbacktitle=black, boxrule=0.2mm,arc=0mm, left = 2pt}

\begin{tcolorbox}[title=\textsc{Coordinator} \vphantom{\texttt{Worker  i}}]
Initialize $\bar x^0$\\
%Send $\overline x$ to each machine\\
\While{test $\mathsf{C}$ not verified}{
    \hspace*{0ex}\\
     {\color{blue!70!black}Receive\,$[\Delta^k]_{\mathbf{S}^{k-D_{i^k}^k}}$\,from\,agent\,$i^k$}\\
     $\barx^k \leftarrow \barx^{k-1} + \alpha_i[\Delta^k]_{\mathbf{S}^{k-D_{i^k}^k}} $\\
     $ x^k \leftarrow \rgprox(\bar x^k )$\\
    Draw sparsity $\mathbf{S}^k$ with prob. $\pvec$\\[0.5ex]
    {\color{red!80!yellow} Send $x^k, \mathbf{S}^k$ to agent $i^k$}\\
    $k\leftarrow k+1$
}
%Interrupt all workers\\
%\textbf{Output} $x^k$
\end{tcolorbox}
\tcbset{width=0.48\columnwidth, colframe=black!50!black, coltitle=white, colbacktitle=black}
\begin{tcolorbox}[title=\textsc{Worker } $i$]
Initialize $x_i = x_i^+=0$\\
\While{not interrupted}{
    \hspace*{0.1ex}\\
     {\color{red!80!yellow}Receive  $x$ and $\mathbf{S}$ from the coordinator}\\
    $[x^+_i]_{\mathbf{S}}   \leftarrow [  x - \gamma \nabla f_i( x)]_{\mathbf{S}} $\\[0.6ex]
    $\Delta \leftarrow x_i^+ - x_i $ \\[0.7ex]
    {\color{blue!70!black}Send $[\Delta]_{\mathbf{S}}$ to the coordinator}\\[0.7ex]
    $[x_i]_{\mathbf{S}} \leftarrow [x^+_i]_{\mathbf{S}} $%\\[0.7ex]
}
\vspace*{0.55cm}
\end{tcolorbox}
\vspace*{-0cm}
\end{algorithm}

The communications per iteration are (i) a blocking {\color{blue!70!black}send/receive} from a worker to the coordinator (in {\color{blue!70!black}blue}) of size $|\mathbf{S}|$, and (ii) a blocking {\color{red!80!yellow}send/receive} from the coordinator to the last updating worker (in {\color{red!80!yellow}red}) of the current iterate.
The upward communication is thus made sparse by the algorithm and the downward communication cost depends on the structure of~$x^k$, which is the output of a proximal operator on $r$. In the case of $\ell_1$-regularization, $x^k$ will become sparse after some iterations, leading to a two-way sparse algorithm; we will be discussed in~Section~\ref{sec:identif}.

%%%%%%%%%%%%%%%%%%%%%%%%%%%%%%%%%%%%%%%%%%%%%%%%%%%
\subsection{Convergence Analysis}\label{sec:analyze}
%%%%%%%%%%%%%%%%%%%%%%%%%%%%%%%%%%%%%%%%%%%%%%%%%%%

We study the convergence properties of \salgo under standard assumptions on\;\eqref{eq:pb} and no apriori assumption on the computing system or on data distribution. We emphasize here that we do not put assumptions on the delays; for instance they do not to need to be bounded or independent of the previous selectors $(\mathbf{S}^k)$. Let us also notice that the local iterates $(x_i)$ do not converge to a minimizer of the individual functions $(f_i)$ but rather to local shifts of the solution $x^\star$ of~\eqref{eq:pb} (unique from the strong convexity assumption).

\begin{theorem}[Reaches and Limits of Sparsification]\label{lm:spy_diff}
Let the functions $(f_i)$ be $\mu$-strongly convex ($\mu>0$) and $L$-smooth. Let $r$ be convex lsc. Take $\gamma\in(0,\frac{2}{\mu+L}]$. Suppose that Assumption~\ref{hyp:algo} holds for the probability vector $\pvec$, and that $ \frac{\pmin}{\pmax} \geq (1-\gamma \mu)^2.$

Then, \salgo on $((\alpha_i),(f_i), r  ~ ; ~  p)$ verifies for all $k\in [k_m, k_{m+1})$
\begin{align}
\label{eq:dif_prob}
   \mathbb{E} \left\| x^k - x^\star \right\|^2 \le \left( \pmax (1 - \gamma \mu)^2 + 1 - \pmin  \right)^{m} \max_i\left\|x_i^0-x_i^\star\right\|^2,
\end{align}
with the shifted local solutions $x_i^\star = x^\star - \gamma_i\nabla f_i(x^\star)$. 

Furthermore, using the maximal stepsize $\gamma = \frac{2}{\mu + L}$, we obtain for all $k\in [k_m, k_{m+1})$
\begin{align}
\label{eq:ratebefore}
 \mathbb{E}   \left\| x^k - x^\star \right\|^2 \le \left( \pmax \left(\frac{1-\kappa_{\eqref{eq:pb}}}{1+\kappa_{\eqref{eq:pb}}}\right)^2 + 1 - \pmin  \right)^{m} \max_i\left\|x_i^0-x_i^\star\right\|^2.
\end{align}
\end{theorem}

\begin{proof}
The co-existence of both deterministic and stochastic delays in the algorithm calls for an original mathematical analysis using the notation introduced in Section\;\ref{sec:prelim}, and in particular the notion of epoch sequence \eqref{eq:km}.
This is detailed in Supplement~\ref{apx:proofsparse}.
\end{proof}

This result establishes bounds that lead to convergence whenever the selection probabilities are well chosen. First, if all probabilities are equal to $1$, the algorithm boils down to \dave and Theorem~\ref{lm:spy_diff} coincides with Theorem~\ref{th:davepg}. In more general cases, this result has to be interpreted more carefully as developed in the following section.

%%%%%%%%%%%%%%%%%%%%%%%%%%%%%%%%%%%%%%%%%%%%%%%%%%%%%%%%%%%%%%%%%%%%%%%%%%%%
\subsection{On the sparsification choice}\label{sec:adapt}
%%%%%%%%%%%%%%%%%%%%%%%%%%%%%%%%%%%%%%%%%%%%%%%%%%%%%%%%%%%%%%%%%%%%%%%%%%%%

In the totally distributed setting, all machines are responsive, which means with our notation:\;$m\to \infty$ when $k\to\infty$. Then, Theorem\;\ref{lm:spy_diff} gives linear convergence of the mean squared error in terms of epochs if 
\begin{equation}\label{eq:prob_gap}
	\frac{\pmin}{\pmax} > (1-\gamma \mu)^2 \stackrel{\gamma = \frac{2}{\mu + L}}{\geq} \left(\frac{1-\kappa_{\eqref{eq:pb}}}{1+\kappa_{\eqref{eq:pb}}}\right)^2
\end{equation}
Thus the behavior of the algorithm depends on if the selection is performed uniformly (and thus structure-blind) or non-uniformly (to encompass some prior information). We distinguish these two cases below.

%%%%%%%%%%%%%%%%%%%%%%%%%%%%%%%%%%%%%%%%%%%%%%%%%%%%%%%%%%%%%%%%%%%%%%%%%%%%
\vspace*{1ex}
\paragraph*{Inefficiency of uniform sparsification}
 
If the selection is \emph{uniform}, i.e. $p_i=\pp\in(0,1]$ for all\;$i$, 
we directly get convergence from  \eqref{eq:dif_prob} as the mean squared error vanishes linearly in terms of epochs with a rate $(1-\pp\gamma\mu(2-\gamma\mu))$, degraded compared to the $(1-\gamma\mu)$ rate of \dave. 
Unfortunately, such uniform selection also results in poor performance in many cases (as illustrated in Supplement\;\ref{apx:num}). 
Thus, when \salgo is used with uniform sampling, it has a degraded performance compared to (non-sparsified) \dave, both in theory and in practice. %(see Figure~\ref{fig:uniform}). 
Adaptivity is key for sparsifying efficiently.

% \begin{figure}
%     \centering

%     \caption{Uniform}
%     \label{fig:uniform}
% \end{figure}

%%%%%%%%%%%%%%%%%%%%%%%%%%%%%%%%%%%%%%%%%%%%%%%%%%%%%%%%%%%%%%%%%%%%%%%%%%%%
\vspace*{1ex}
\paragraph*{Efficiency of adaptive sparsification}

There exist several adaptation strategies for selecting coordinates. In the context of coordinate descent, we can mention greedy coordinates selection~\cite{dhillon2011nearest,nutini2017let}, other heuristics~\cite{loshchilov2011adaptive,glasmachers2013accelerated}, and 
importance sampling~\cite{zhao2015stochastic,richtarik2016optimal}, that can sometimes be used in practice.
The idea of adaptive coordinate descent methods based on the coordinate-wise Lipschitz constants and current values of the gradient is proposed in 
\cite{perekrestenko2017faster,namkoong2017adaptive,stich2017safe}.
The idea of adaptively using the \emph{iterates structure} enforced by a non-smooth regularizer was recently developed in \cite{MOR}. 

We use here the technique of \cite{MOR} specialized for the coordinate selection: when some coordinates get null, there is some hope that they will  remain null for subsequent iterations, and it is thus natural to
update preferentially the non-null coordinates. Mathematically, this means that we select coordinates in the active support as follows:
\begin{align*}
\mathbb{P}[j\in\mathbf{S}^k] = \left\{ \begin{array}{cr}
   \pp  & \text{ if } x^k_{[j]} = 0 \\
   1  &  \text{ if } x^k_{[j]} \neq  0 
\end{array}   \right. ~~~~~ \text{for all $j\in\{1,\dots,d\}$ ~and~ } \pp \in(0,1].
\end{align*}
In words, we communicate
%this means updating/communicating 
the coordinates in the support of the coordinator point $x^k$, together with coordinates outside the support, randomly selected with some \emph{exploration probability}~$\pp$.

This adaptive sampling often shows tremendous gains in practice compared to uniform sampling; however, it may not converge in some situations. This is due to two technical points related to Theorem~\ref{lm:spy_diff}:
\begin{itemize}
    \item A \emph{good conditioning} is necessary to allow for a small $\pp=\pmin$ (with $\pmax=1$). More precisely, from~\eqref{eq:prob_gap}, we get that the minimal conditioning to allow for a probability\;$\pp$ of selection outside the support is:
    \begin{align}
        \label{eq:kappamin}
        \kappa_{\eqref{eq:pb}} > \kappa_{\min} := \frac{1-\sqrt{\pp}}{1+\sqrt{\pp}} .
    \end{align}
    Since we aim at taking $\pp$ small to communicate little, this is a stringent condition.
    \item The sampling is \emph{not i.i.d.}\;anymore since the probabilities depend on the points generated by the algorithm. This difficulty could be tackled with a small algorithmic fix that unfortunately degrades practical performances and a refined analysis (see the discussion to Supplement~\ref{sec:noniid}), but the upcoming methods directly address this point.
\end{itemize}
While these two issues appear separate, they can  both be overcome by iteratively reconditioning the problem. %This is what will be developed in the following section. 

% ===========================================================
% ===========================================================
\section{Proximal reconditioning for adaptive sparsification}\label{sec:recon}
% ===========================================================
% ===========================================================

The learning problem~\eqref{eq:pb} should be well-conditioned to safely apply the random sparsification technique of \salgo with reasonable exploration probability\;$\pi$. The idea is then not to apply \salgo directly to~\eqref{eq:pb} but rather to a modified problem for which we control the condition number and the sparsification potential. 

We thus propose to recondition the learning problem~\eqref{eq:pb} using the standard proximal algorithm (see e.g.\;\cite{rockafellar1976monotone}), as described in Section\;\ref{sec:proxrecon}. We present in Section\;\ref{sec:recoalgo} our algorithmic choices and the resulting algorithm, called \recoalgo.

% ===========================================================
\subsection{Proximal reconditioning}\label{sec:proxrecon}
% ===========================================================

This type of methods consist in iteratively regularizing the problem with the squared distance to some center point. We call \emph{outer iteration} the process of (approximatively) solving such a reconditioned problem. At outer loop\footnote{The quantities related to outer loop $\ell$ are denoted with a subscript $\ell$.} $\ell$, we define the worker $i$'s regularized function as 
\begin{align*}
    h_{i,\ell} =  f_i + \frac{\rho}{2} \| \cdot - x_\ell \|_2^2 
\end{align*}
where $\rho$ is the regularization factor and $x_\ell$ the center point at outer loop $\ell$. The \emph{reconditioned problem for loop $\ell$} then writes
\begin{align}\label{eq:reco}
\tag{$\mathsf{R}_\ell$}
& \min_{x\in\mathbb{R}^d}  ~~ H_\ell(x)  :=  \sum_{i=1}^M  \alpha_i \underbrace{\left( f_i(x) + \frac{\rho}{2} \| x - x_\ell \|_2^2 \right)}_{h_{i,\ell}(x)}  +  r(x) .
\end{align}
For $\mu$-strongly convex $L$-smooth $(f_i)$, the regularized functions $h_{i,\ell}$ are $(\mu+\rho)$-strongly convex and $(L+\rho)$-smooth. Hence, the condition number of the smooth part of~\eqref{eq:reco} writes
\begin{align*}
   \kappa_{\eqref{eq:reco}} = \frac{\mu + \rho}{L+\rho}\, ~~\Big(\geq~\kappa_{\eqref{eq:pb}} = \frac{\mu}{L}\Big).
\end{align*}
The optimal solution of\;\eqref{eq:reco} is exactly the proximal point~\eqref{eq:prox} of $F/\rho$ at $x_\ell$
\begin{align*}
\prox_{F/\rho} (x_\ell) = \argmin_{x\in\mathbb{R}^d}  ~ \underbrace{  \sum_{i=1}^M \alpha_i f_i(x) +  r(x)}_{= F(x)} +  \frac{\rho}{2} \| x- x_\ell \|_2^2.
\end{align*}
Thus, for solving~\eqref{eq:pb}, each outer iteration consists in an (inexact) proximal step:
\begin{align}
\label{eq:approx}
    x_{\ell+1} \approx \prox_{F/\rho} (x_\ell) 
\end{align}

The proximal point algorithm is a standard regularization approach in 
optimization. It was presented in \cite[Chap.\,5]{bkl} to recondition a convex quadratic objective, for which computing the proximal operator~\eqref{eq:prox} is easy (it is the unique solution of a linear system, well-conditioned by construction). The general proximal algorithm was then popularized by the seminal works\;\cite{martinet-1970,rockafellar1976monotone}. The study of these algorithms, and especially their inexact variants, has attracted a lot of attention; see e.g.\;\cite{guler1992new,solodov2000error,fuentes,lin2017catalyst,lin2019inexact}. 

Practical implementations of such methods require an inner algorithm to compute the proximal point (we will use \salgo here) and a rule to stop this algorithm. Several papers consider the key question of inner stopping criteria for inexact proximal methods in various contexts; see e.g.\;\cite{fuentes} in smooth optimization, \cite{lemarechal-sagastizabal-1997} in nonsmooth optimization,
and \cite{solodov-svaiter-2001} in operator theory. In this paper, we use the standard criteria, following \cite{rockafellar1976monotone}, which gave the first inexact rules for the proximal point algorithm in the context of monotone operators.

% ===========================================================
\subsection{\recoalgo}\label{sec:recoalgo}
% ===========================================================

We present our main algorithm, which consists in applying the inexact proximal scheme\;\eqref{eq:approx} with \salgo as inner algorithm to solve \eqref{eq:pb}.

At the outer iteration $\ell$, we run \salgo for solving \eqref{eq:reco} with i.i.d.\;non-uniform sparsification probabilities given, for a fixed $0<c\leq d$, by
\begin{align}
\label{eq:pvec}
p_{j,\ell} = \left\{ \begin{array}{cr}
 \pp_\ell := \min\left(\displaystyle\frac{c}{|\nullC(x_\ell)|};1\right) & \text{ if } (x_\ell)_{[j]} = 0 \\
   1  &  \text{ if } (x_\ell)_{[j]} \neq  0 
\end{array}   \right. ~~~~ \text{for all $j\in\{1,\dots,d\}$}.
\end{align}
The sparsification level over outer iterations is then bounded from below by
\begin{align*}
     \pp := \frac{c}{d} ~\leq~ \inf_\ell \pp_\ell.
\end{align*}

We now choose the reconditioning parameter $\rho$ from $\pp$ so that \salgo converges linearly to the solution of the reconditioned problem~\eqref{eq:reco}. We know, from Section~\ref{sec:adapt}, that this is the case as soon as
\begin{align*}
   \kappa_{\eqref{eq:reco}} = \frac{\mu + \rho}{L+\rho}  > \kappa_{\min} ~~\Longleftrightarrow~~  \rho > \frac{\kappa_{\min}L-\mu}{1-\kappa_{\min}} \quad\text{ with } \kappa_{\min} = \frac{1-\sqrt{\pp}}{1+\sqrt{\pp}}
   \text{~as in~\eqref{eq:kappamin}}.
\end{align*}
% Not Fundamental:
% %\begin{remark}[On the choice of $\rho$]
% For simplicity, we choose to present here only the case where the regularization $\rho$ is constant over all outer loops. It is possible to take varying ones, to almost no change but i) it adds another moving part; and ii) the recovered rate is ultimately only slightly (if even) better. 
% %\end{remark}
To properly handle the strict inequality above, we propose to choose a conditioning which guarantees a $(1-\alpha)$ rate for \salgo on the reconditioned problems uniformly over $\ell$. Mathematically, for $0<\alpha<\pp$ (for instance $\alpha=\pp/2$), we choose
\begin{align}
\label{eq:rho}
    \rho = \frac{\kappa_{\eqref{eq:reco}} L-\mu}{1-\kappa_{\eqref{eq:reco}}} \quad\text{ with }\quad \kappa_{\eqref{eq:reco}} = \frac{1-\sqrt{\pp-\alpha}}{1+\sqrt{\pp-\alpha}} .
\end{align}
Then, the contraction factor of \salgo with the maximal stepsize (see~\eqref{eq:ratebefore}) for the reconditioned problem \eqref{eq:reco} becomes
% Not fundamental:
% \footnote{If the problem is already well-enough conditioned (i.e.~\eqref{eq:kappamin} holds), $\rho$ might be negative and thus would imply \emph{deconditioning}. For the sake of simplicity, we will only consider in this section the most common case where the original problem is not well-enough conditioned.} 
\begin{align}
\label{eq:reco_contra}
     \left( \left(\frac{1-\kappa_{\eqref{eq:reco}}}{1+\kappa_{\eqref{eq:reco}}}\right)^2 + 1 - \pi_\ell  \right) =  \left( \pi - \alpha + 1 - \pi_\ell  \right) = \underbrace{(1 - \alpha - (\pi_\ell - \pi) )}_{=: (1-\alpha_\ell)} \leq 1-\alpha < 1 .
\end{align}
This means that \salgo is linearly convergent on the reconditioned problem~\eqref{eq:reco}. Thus, it can safely be used as an inner method in the inexact proximal algorithm \eqref{eq:approx} to solve the original problem~\eqref{eq:pb}. 

The remaining part is to the choice of a stopping criterion for the inner loop. We propose to use three different criteria: epoch budget, absolute accuracy, and relative accuracy (called $\mathsf{C}_1, \mathsf{C}_2$, and $\mathsf{C}_3$ respectively). Stopping criteria based on accuracy are usually more stringent to enforce (see e.g.\;\cite[Sec.~2.3]{lin2017catalyst} and references therein), however they may bring significant performance improvement when the instantaneous rate is better than the theoretical one. 

The resulting algorithm, called\;\recoalgo, is presented as Algorithm\;\ref{algo:reco}. Under any of the three stopping criteria, we recover the same convergence result, formalized in the next theorem.

% ===================================
\begin{algorithm}[h!]
\caption{\label{algo:reco}\recoalgo on $((\alpha_i),(f_i),r)$}
\begin{flushleft}
Initialize $x_1$, $d\geq c>0$, and $\delta\in(0,1)$.
\begin{align}
\label{eq:set}
 \text{Set } \rho = \frac{\kappa L-\mu}{1-\kappa} \text{ and } \gamma \in \left( 0, \frac{2}{\mu+L+2\rho} \right] \text{ with } \kappa = \frac{1-\sqrt{\pp -\alpha}}{1+\sqrt{\pp-\alpha}}; \pp = \frac{c}{d} \text{ and } \alpha = \frac{c}{2d}.
\end{align}
\While{the desired accuracy is not achieved}{
Observe the support of $x_\ell$, compute $\pvec_\ell$ as
\begin{align}
\label{eq:proba}
p_{j,\ell} = \left\{ \begin{array}{cr}
 \pp_\ell := \min\left(\frac{c}{|\nullC(x_\ell)|};1\right) & \text{ if } [x_\ell]_j = 0 \\
   1  &  \text{ if } [x_\ell]_j \neq  0 
\end{array}   \right. ~~~~~ \text{for all $j\in\{1,\dots,d\}$}.
\end{align}
Compute an approximate solution of the reconditioned problem
\begin{align}
\label{eq:approxalgo}
& x_{\ell+1} \approx \prox_{F/\rho} (x_\ell) =\argmin_{x\in\mathbb{R}^d}  ~ \left\{ %H_\ell(x) =  
\sum_{i=1}^M  \alpha_i \underbrace{\left( f_i(x) + \frac{\rho}{2} \| x - x_\ell \|_2^2 \right)}_{h_{i,\ell}(x)}  +  r(x) \right\} 
\end{align}
with \salgo on $\Big((\alpha_i),(h_{i,\ell}), r  \;;\;  \pvec_\ell\Big)$ with $x_\ell$ as initial point and with the stopping criterion:\\[0.1cm]
\begin{itemize}
    \item[] $\mathsf{C}_1$ (epoch budget): Run \salgo with the maximal stepsize for
    $$  \displaystyle \mathsf{M}_\ell = \left\lceil \frac{ (1+\delta) \log(\ell)}{\log\left( \frac{1}{1-\alpha+\pp - \pp_\ell} \right)} + \frac{\log\left(\frac{2 \mu+\rho}{(1-\delta)\rho}\right)}{\log\left( \frac{1}{1-\alpha+\pp - \pp_\ell} \right)} \right\rceil \text{ epochs.}$$
    \item[or] $\mathsf{C}_2$ (absolute accuracy):  Run \salgo until it finds $x_{\ell+1}$ such that  
    $$  \|x_{\ell+1} -  \prox_{F/\rho}(x_\ell)\|^2 \leq \frac{(1-\delta)\rho}{(2\mu+\rho) \ell^{1+\delta}} \|x_{\ell} -  \prox_{F/\rho}(x_\ell)\|^2  .$$ 
    \item[or] $\mathsf{C}_3$ (relative accuracy):  Run \salgo until it finds $x_{\ell+1}$ such that  
    $$  \|x_{\ell+1} -  \prox_{F/\rho}(x_\ell)\|^2 \leq \frac{\rho}{4(2\mu+\rho) \ell^{2+2\delta}} \|x_{\ell+1} -  x_\ell \|^2  .$$ 
\end{itemize}
}
\end{flushleft}
\end{algorithm}

\begin{theorem} 
\label{th:reco}
Let the functions $(f_i)$ be $\mu$-strongly convex ($\mu\geq0$) and $L$-smooth. Let $r$ be convex lsc. If $\mu=0$ and $\mathsf{C}_3$ is used, we furthermore require that $F$ has a unique minimizer $x^\star$ and that $\liminf_{x\to x^\star} (F(x)-F(x^\star))/\|x-x^\star\|^2 > 0$.

Then, the sequence generated by \recoalgo on $((\alpha_i),(f_i),r)$ with stopping criterion $\mathsf{C}_1, \mathsf{C}_2$, or $\mathsf{C}_3$ converges almost surely to a minimizer of $F$.
Furthermore, if $\mu>0$, then we have\footnote{We use the standard notation: $a_\ell = \tilde{\mathcal{O}} ( (1-r)^\ell)$ denotes that there exists $C,p$ such that $ a_\ell \leq C \ell^p (1-r)^\ell $. } 
      \begin{align*}
        \EE \left[   \left\|x_{\ell+1} - x^\star  \right\|^2 \right] = \tilde{\mathcal{O}} \left( \left( 1 - \frac{{\mu }}{\mu + {\rho }/2} \right)^\ell \right) \quad\text{ for criterion } \mathsf{C}_1; \\
      \left\|x_{\ell+1} - x^\star  \right\|^2  = \tilde{\mathcal{O}} \left( \left( 1 - \frac{{\mu }}{\mu + {\rho }/2}  \right)^\ell \right)  \quad\text{ for criteria } \mathsf{C}_2, \mathsf{C}_3.
        \end{align*}
\end{theorem}

\begin{proof}
Even though the final result is similar for the three cases, the proof techniques are rather different. We thus present them separated in Supplement~\ref{apx:reco}.
\end{proof}

This result thus establishes that \recoalgo converges linearly to a solution of \eqref{eq:pb}. This means that \recoalgo has qualitatively the same behavior as \dave, with the additional feature of having sparse local updates and therefore sparse upward communications. In other words, our algorithm is similar to the baseline in terms of iterations, but it is expected to be faster in terms of communications (more precisely in terms of quantity of information exchanged between coordinator and workers) which would result in a wallclock gain in practice, as shown in Section\;\ref{sec:num}. Before this, we further investigate in the next section the theoretical gain of our sparsification technique in the case of sparse optimal solutions.

\begin{remark}[Acceleration (with respect to iterations)]\label{rem:cata}
In this paper, we are interested in sparsifying communications and we primarily consider the reconditioning aspect of proximal methods, leaving aside other aspects including acceleration. As proposed in\;\cite{guler1992new}, the iterations of the inexact proximal algorithm can indeed be accelerated using Nesterov's method\;\cite{nesterov1983method}.
The recent works\;\cite{lin2017catalyst,lin2019inexact} also propose accelerated and quasi-Newton variants of the inexact proximal point algorithm as a meta-algorithm to improve the convergence of optimization methods (driven by machine learning applications \cite{lin2015universal}). 
%In this paper,
Here, we investigate the complexity in terms of communications rather than iterations, so we do not insist much on these accelerated variants. %However 
The developments of this section could still be extended to accelerated proximal algorithm, following the meta-algorithm of \cite{lin2017catalyst}. We briefly study this direction in Supplement\;\ref{apx:cata}.  
In particular, Theorem~\ref{thm:cata} shows a gain of the square root on the rate for the accelerated version compared to Theorem~\ref{th:reco}, as usually observed with accelerated methods. 
\end{remark}

% ===========================================================
% ===========================================================
\section{Identification for two-way sparse communications} 
% ===========================================================
% ===========================================================

In the previous sections, we present an adaptive sparsification of upward communications (worker-to-coordinator) and show that the resulting algorithm converges after proximal reconditioning. By construction, the downward communications (coordinator-to-workers) depends on the structure of~$x^k$ (the coordinator point of the inner method), which is the output of a proximal operator on $r$. In the case of $\ell_1$-regularization or other sparsity-promoting regularization\;\cite{bach2012optimization}, we show in Section\;\ref{sec:identif} that the $x^k$ eventually become sparse after some iterations. This automatically makes our algorithm a two-way sparse algorithm. 
Finally, in Section\;\ref{sec:comm}, we take a closer look to the complexity of our algorithm with respect to the communication cost.

For this study, we make an additional assumption that our problem has a strongly sparse solution. This assumption is divided into two parts: i) the regularizer $r$ should induce a \emph{stable} support at the optimum (through its proximity operator); and ii) this optimal support $\supp(x^\star)$ should be small with respect to the ambient dimension.
\begin{assumption}[Strongly sparse optimal solution] \label{hyp:ident}
Problem~\eqref{eq:pb} is $\mu$-strongly convex ($\mu>0$) and its solution $x^\star$ verifies
\begin{itemize}
   \item[i)] $\exists~ \varepsilon>0$ such that $\supp(x^\star) = \supp \left( \prox_{r} \left( x^\star - \sum_{i=1}^M  \alpha_i  \nabla f_i(x^\star) + \mathbf{e}   \right)   \right)~ \forall \mathbf{e} \in \mathcal{B}(0,\varepsilon) $;
    \item[ii)] the size $s^\star =|\supp(x^\star)|$ of the optimal support is small compared to $d$: $s^\star\ll d$.
\end{itemize}
\end{assumption}
While part ii) is rather explicit, part i) is quite abstract. For instance, when $r= \lambda \|\cdot\|_1$, using the explicit form of the proximity operator\;\eqref{eq:soft}, part i) directly translates to
\begin{align*}
   \sum_{i=1}^M  \alpha_i  \nabla_{[j]} f_i(x^\star)  \in (-\lambda,\lambda)  ~~~ \text{ for all } j \in \nullC(x^\star). 
\end{align*}
This condition matches the nondegeneracy condition for sparse solutions commonly admitted for exact recovery in machine learning; see e.g.~\cite{nutini2019active,pmlr-v89-sun19a}.
The interest of the general assumption i) is that it accounts for a variety of sparsity-inducing regularizations, including weighted $\ell_1$-norms, ``group'' $\ell_1/\ell_q$-norms; see \cite[Sec.~3.3]{bach2012optimization}.

% ===========================================================
\subsection{Identification and consequences} \label{sec:identif}
% ===========================================================

The iterates of proximal algorithms usually \emph{identify} the optimal structure; see e.g\;\cite{vaiter2015low} or \cite{iutzeler2020SVAA}. In the case of $\ell_1$-regularization, this means that proximal algorithms produce iterates that eventually have the same support as the optimal solution of \eqref{eq:pb}. Unfortunately, randomness may break this identification. For instance, it is well-known that for the proximal stochastic gradient descent, the sparse structure may not be identified with probability one; see e.g.\;\cite{lee2012manifold} and a counter-example in\;\cite{poon2018local}. We first establish that our algorithm\footnote{Note that a similar identification result holds for \dave: in the same context, we have  $\supp(x^k) = \supp(x^\star)$ for $k\geq k_m$ and $m$ larger than some threshold; see Supplement~\ref{apx:ident}.} does identify the optimal support under the non-degeneracy assumption~\ref{hyp:ident}.

% All existing results (including \cite{lee2012manifold}) use a strong non-degeneracy assumption to establish identification results. The only exception is \cite{fadili2018sensitivity} that presents extended identification results without non-degeneracy for a large class of nonsmooth regularizers. Identification of our distributed algorithm in the $\ell_1$ regularization is based on these general identification results.% of \cite{fadili2018sensitivity}.
% To the best of our knowledge, this is the first time that such a property is shown for an asynchronous algorithm.

%\begin{assumption}[Additional assumption for identification]
%\label{hyp:delident}
%The number of iterations between two full updates cannot grow exponentially, i.e. $k_{m+1}-k_m = o(\exp(m))$. This assumption is rather mild and subsumes the usual bounded delay assumption.
%\end{assumption}

\begin{theorem}[Identification]\label{th:ident}
Let the functions $(f_i)$ be $\mu$-strongly convex and $L$-smooth. Let $r$ be convex lsc. Under Assumption~\ref{hyp:ident}, the \emph{outer and inner} iterates of \recoalgo identify the optimal structure in finite time: with probability one there exists a finite time $\Lambda<\infty$ such that
$$ 
\supp(x^k_\ell) = \supp(x_\ell)  = \supp(x^\star) ~~~ \text{ for any } k \text{ and all }  \ell \geq \Lambda
$$
where $x^k_\ell$ denotes the $k$-th iterate produced by \salgo during the $\ell$-th outer loop. 
\end{theorem}

\begin{proof}
We proved in the previous section that \recoalgo converges almost surely, but it is not enough to guarantee identification in general\footnote{Take $d=1$, $F(x) = |x|$, and $x_{\ell+1} = \prox_{|\cdot|}(x_\ell) + 1/\ell^2$. The minimum of $F$ is $0$ but we have $x_\ell = 1/(\ell-1)^2>0$ for all $\ell>1$.}. Here it is the fact that the inner algorithm \salgo features a proximity operator with a non-vanishing stepsize that yields the following identification property. This proof is detailed in Supplement~\ref{apx:ident}.
\end{proof}

This identification has two consequences on communications in our distributed setting. First, identification implies that the variables communicated by the coordinator to the workers will eventually be sparse. Second, this sparsity is also leveraged in the sparsification strategies of \recoalgo where only the coordinates in $\null(x_\ell)$ are randomly zeroed. Thus, for sparsity inducing problems such as $\ell_1$-regularized learning problems, our distributed algorithm has, structurally, \emph{two-way sparse communications}.

Even better, once this identification occurs, the rate of the inner algorithm \salgo dramatically improves to match the rate of its non-sparsified version \dave. To get this improved rate, a small additional property is needed on the regularizer: $r$ has to be separable with respect to $\supp(x^\star)$ i.e. $r(x) = r_1([x]_{\supp(x^\star)})+r_2([x]_{\nullC(x^\star)})$ which holds true for almost all sparsity inducing regularizations \cite[Sec.~3.3]{bach2012optimization}.

\begin{theorem}[Improved rate]\label{th:imprate}
Let the functions $(f_i)$ be $\mu$-strongly convex and $L$-smooth. Let $r$ be convex lsc, separable with respect to $\supp(x^\star)$. Under Assumption~\ref{hyp:ident}, the \emph{inner} iterates of \recoalgo benefit from an improved rate after identification. There is $\Lambda<\infty$ such that for all $\ell>\Lambda$ and  $k\in [k_m, k_{m+1})$,
using the maximal stepsize $\gamma = \frac{2}{\mu + L + 2\rho}$,
\begin{align*}
    \left\| x_\ell^k - x_\ell^\star \right\|^2 \le \left( \frac{1-\kappa_{\eqref{eq:reco}}}{1+\kappa_{\eqref{eq:reco}}}  \right)^{2m} \left\|x_{\ell}-x_\ell^\star\right\|^2.
\end{align*}
where $x^k_\ell$ denotes the $k$-th iterate produced by \salgo during the $\ell$-th outer loop.

\end{theorem}

% \begin{theorem}[Improved rate]\label{th:imprate}
% Let the functions $(f_i)$ be $\mu$-strongly convex and $L$-smooth. Let $r$ be convex lsc, separable with respect to $\supp(x^\star)$. Under Assumption~\ref{hyp:ident}, the \emph{inner} iterates of \recoalgo benefit from an improved rate after identification. There is $\Lambda<\infty$ such that for all $\ell>\Lambda$ and  $k\in [k_m, k_{m+1})$
% \begin{align*}
%     \left\| x_\ell^k - x_\ell^\star \right\|^2 \le \Big(1 - \gamma (\mu+\rho)  \Big)^{2m} \left\|x_{\ell}-x_\ell^\star\right\|^2
% \end{align*}
% where $x^k_\ell$ denotes the $k$-th iterate produced by \salgo during the $\ell$-th outer loop.

% Furthermore, using the maximal stepsize $\gamma = \frac{2}{\mu + L + 2\rho}$, we obtain for all $k\in [k_m, k_{m+1})$
% \begin{align*}
%     \left\| x_\ell^k - x_\ell^\star \right\|^2 \le \left( \frac{1-\kappa_{\eqref{eq:reco}}}{1+\kappa_{\eqref{eq:reco}}}  \right)^{2m} \left\|x_{\ell}-x_\ell^\star\right\|^2.
% \end{align*}
% \end{theorem}

\begin{proof}
    The proof is reported to Supplement~\ref{apx:imprate}.
\end{proof}

Thus our algorithm eventually has the practical interest of having sparse two-way communication, at almost no additional cost.

% ===========================================================
\subsection{Communication complexity}\label{sec:comm}
% ===========================================================

We study in this section the asymptotic communication complexity of our method in terms of \emph{number of coordinates (\emph{real} numbers) exchanged between the coordinator and the workers}. To do so, we combine the number of outer iterations to reach an accuracy of $\varepsilon$ with the asymptotic communication cost of the inner iterations needed to reach the stopping criteria. Mathematically, we define
\begin{align}\label{eq:commcomp}
    \mathbf{C}(\varepsilon) = ( \mathbf{c}^{\mathrm{up}} + \mathbf{c}^{\mathrm{down}} ) K \mathbf{M}^{\mathsf{C}} \mathbf{L}(\varepsilon) 
\end{align}
where 
\begin{itemize}
    \item $\mathbf{c^{\mathrm{up}}}$ (resp.\;$\mathbf{c^{\mathrm{down}}}$) is the (expected) number of coordinates communicated from the coordinator to the active worker (resp. from the active worker to the coordinator) during one iteration and $K$ is the average number of iterations per epoch;
    \item $\mathbf{M}^{\mathsf{C}}$ is the (expected) number of inner epochs to reach stopping criterion $\mathsf{C}$;
    \item $\mathbf{L}(\varepsilon)$ is the number of outer loops to reach accuracy $\varepsilon$.
\end{itemize}

Focusing on the final regime of the algorithm when identification has taken place (as per Section~\ref{sec:identif}), we get for \recoalgo
\begin{align*}
    \mathbf{c^{\mathrm{up}}} = |\supp(x_\ell^k)| = |\supp(x^\star)| = s^\star \quad\text{ and }\quad \mathbf{c^{\mathrm{down}}} = s^\star + c
\end{align*}
which leads to the following communication complexity for \recoalgo.

\begin{theorem}[Communication complexity of \recoalgo]\label{th:com}
Let the functions $(f_i)$ be $\mu$-strongly convex and $L$-smooth ($\mu>0$), $r$ be convex lsc, separable with respect to $\supp(x^\star)$. Let Assumption \ref{hyp:ident} hold.
If the parameter $c$ is of the same order as $s^\star$ compared to $d$ ($c\approx s^\star \ll d$), then the communication complexity \eqref{eq:commcomp} of \recoalgo with criteria $\mathsf{C}_2$ or $\mathsf{C}_3$ is:
\begin{align*}
     \mathbf{C}(\varepsilon) = \tilde{O}\left( \frac{L-\mu}{\mu} \sqrt{d \, s^\star } ~ \max \left\{\sqrt{\frac{c}{s^\star}};\sqrt{\frac{s^\star}{c}} \right\} \log\left( \frac{1}{\varepsilon} \right) \right).
\end{align*}
\end{theorem}

\begin{proof}
The proof consists in evaluating the terms in \eqref{eq:commcomp}, one by one, in the right regime. It is given in Supplement~\ref{apx:com}.
\end{proof}

In order to show the benefit in terms of communications of the proposed method, we can compare this result with our baseline \dave. For \dave, there is no inner loop so $\mathbf{M}^{\mathsf{C}}=1$ and Theorem~\ref{th:davepg} gives us  $\mathbf{L}(\varepsilon) = \mathcal{O}((\mu+L)/\mu\log(1/\varepsilon))$.  However, even if \dave identifies which implies that $\mathbf{c^{\mathrm{up}}} = s^\star$ as previously, the cost of an upward communication is $\mathbf{c^{\mathrm{down}}} = d$ since there is no sparsification.  This yields the following gain in communication complexity of our algorithm over \dave%\recoalgo over \dave 
($1$ meaning similar performances; the greater, the better):
\begin{align*}
      \tilde{O}\left( \frac{1+\kappa_{\eqref{eq:pb}}}{1-\kappa_{\eqref{eq:pb}}}   ~ \min \left\{\sqrt{\frac{c}{s^\star}};\sqrt{\frac{s^\star}{c}} \right\} ~ \frac{d+s^\star}{\sqrt{ds^\star}}  \right).
\end{align*}
% \begin{corollary}[Communication complexity speedup of \recoalgo vs \dave]\label{cor:comcomp}
% Let the functions $(f_i)$ be $\mu$-strongly convex and $L$-smooth ($\mu>0$). Let $r$ be convex lsc. Let assumptions \ref{hyp:ident} and \ref{hyp:delident} hold.
% If the parameter $c$ is of the same order as $s^\star$ compared to $d$ ($c\approx s^\star \ll d$), then, in terms of communication complexity, the gain of \recoalgo over \dave is 
% \begin{align*}
%       \tilde{O}\left( \frac{1+\kappa_{\eqref{eq:pb}}}{1-\kappa_{\eqref{eq:pb}}}   ~ \min \left\{\sqrt{\frac{c}{s^\star}};\sqrt{\frac{s^\star}{c}} \right\} ~ \frac{d+s^\star}{\sqrt{ds^\star}}  \right).
% \end{align*}
% \end{corollary}
The gain shows a product of three terms.
The first one is greater than $1$ and depends on the conditioning; the second one is in $(0,1]$ but should be not far from $1$, provided that the final sparsity is not too poorly estimated. Finally, the last term fully exhibits the merits of adaptive sparsification with a term in $d + s^\star $ for \dave which is much greater than the $\sqrt{ds^\star}$ for \recoalgo. This last term thus shows a nice dependence in the dimension of the problem and optimal solution for the proposed method. This theoretical gain of \recoalgo compared to \dave is confirmed in the next numerical illustrations.

% ===========================================================
% ===========================================================
\section{Numerical illustrations}\label{sec:num}
% ===========================================================
% ===========================================================

% \section{Numerical illustrations}\label{sec:spy_exps}

In this section, we illustrate the communication gain provided by adaptive sparsification, on two classic $\ell_1$-regularized empirical risk minimization problems.

%%%%%%%%%%%%%%%%%%%%%%%%%%%%%%%%%%%%%%%%%%%%%%%%%%%
\medskip
\noindent
\textbf{Problems.} We consider two types of problems: lasso with randomly generated data and $\ell_1$-regularized logistic regression on popular datasets.
We first consider a lasso problem %of the form
$$
\min_{x\in \RR^d} ~~\| Ax - b\|^2 + \lambda_1 \|x \|_1
$$
with $d=1000$ features and %two different sizes of example set $m= 500$ and $m = 10,000$. 
an example set of size $m= 500$.
We take $A$ randomly generated from the standard normal distribution, $b = Ax_0 +e$ where $x_0$ is a $99\%$ sparse vector and $e$ is taken from the normal distribution with standard deviation $0.01$.  
%\end{align*}
We take $\lambda_1$ to reach %the sparsity of 90\%.
%For this problem, 
an optimal solution of size $| \supp(x^\star) | = 12$ (i.e.\;1.2\% of $d=1000$).

We also examine the regularized logistic regression with elastic net
%"epsilon" from LibSVM data set with 100,000 observations and 2,000 features. 
% estimated over a training set $\mathcal S=\{(z_j,y_j), j\in\{1,\ldots,n\}\}$
%\begin{align*}%\label{eq:reg-loss}
 $$    \min_{x\in \RR^d}~~\frac{1}{m}\!\sum_{j=1}^m \log(1 \!+\! \exp(-y_j z_j^{\top} x)) + \lambda_1\!\left\|x\right\|_1 \!+ \frac{\lambda_2}{2}\!\left\|x\right\|_2^2
 $$
%\end{align*}
on two standard datasets from the LibSVM repository: 
\emph{madelon} ($d=500$ $m=2000$) \emph{rcv1\_train}($d=47236$ $m=20242$). We take
hyperparameters as follows: for madelon, $\lambda_2 = 0.001$ and $\lambda_1 = 0.03$ chosen to reach $99\%$ sparsity (i.e.
%optimal solution with 
%For this problem, 
$| \supp(x^\star) | = 5$); % (i.e. 1\% of $d=500$);
for rcv1, $\lambda_2 = 0.0001$ and $\lambda_1  = 0.001$ to reach $99.7\%$ sparsity (i.e.\;$| \supp(x^\star) | = 61$). %i.e. 0.129\% of $d=47236$.

% %%%%%%%%%%%%%%%%%%%%%%%%%%%%%%%%%%%%%%%%%%%%%%%%%%%
% \medskip
% \noindent
% \textbf{Experimental set-up.} 
%exploit in use sparsity
% To communicate sparse vectors, we send a list of coordinates then their values, as usual in sparse communications. 
We distribute these problems on 
%We run our experiments on 
a machine with $32$ cores and $256$ Gb of RAM: one core plays the role the coordinator while $M$ cores act as workers ($M=5$ for LASSO problems, $M=10$ for madelon, and $M=20$ for rcv1). The data sets are split evenly between the $M$ workers.%, each having access only to its own part.% which rules out using parallel algorithms.
%The code is in Python and uses MPI %(Message Passing Interface) 
%for the communications framework. 

%%%%%%%%%%%%%%%%%%%%%%%%%%%%%%%%%%%%%%%%%%%%%%%%%%%
\medskip
\noindent
\textbf{Algorithms.} We illustrate our sparsified algorithm
\recoalgo (Algorithm~\ref{algo:reco}) for different the amount of randomly chosen coordinates $c$. We take a simplified stopping criteria $\mathsf{C}_1$ with %that consider
$\mathsf{M}_\ell=1$; we thus stop the inner iterations after two passes over the data, following the practical guidelines of Catalyst\;\cite{lin2017catalyst}. We observe that this simple stopping rule gives similar empirical convergence as $\mathsf{C}_3$ with respect to both iterations and scalars exchanged (without the additional computational cost of the test); for a numerical illustration, see Supplement\;\ref{apx:C3}.

% \emph{Stopping criteria}  $\mathsf{C}_1$ is simpler but overly pessimistic, furthermore $\mathsf{C}_2$ or $\mathsf{C}_3$ should be used to fully gain from identification. $\mathsf{C}_2$ is hard to implement but $\mathsf{C}_3$ can be checked using a \emph{full} gradient evaluation (see e.g. \cite{rockafellar1976monotone}); hence, though it is possible to possible to implement it (leading to good performances, see Fig.~\ref{fig:c3}), it breaks the asynchronous nature of method. The rationale with using $1$ epoch is to get close to what is observed when running $\mathsf{C}_3$ and to match the principle of ``1 pass over the data = 1 restart'' used in Catalyst \cite{lin2017catalyst}. 

We display the performances of the algorithms in three ways:
\begin{itemize}
    \item size of support vs number\footnote{Identification is illustrated plotting the size of $\supp(x^k)$. For \recoalgo, $x^k$ refers to $x^j_\ell$ where there have been $k$ inner iterations during the $\ell-1$ first outer loops plus $j$ in the $\ell$-th inner loop.} of inner iterations, showing the identification 
    \item functional suboptimality vs number of inner iterations, % (amount of epochs \eqref{eq:km}),
     \item functional suboptimality vs communication cost, modelled as the number of couples (coordinate, value) sent from and to the coordinator.
\end{itemize}

\begin{figure}[t!]
\includegraphics[width=0.7\textwidth]{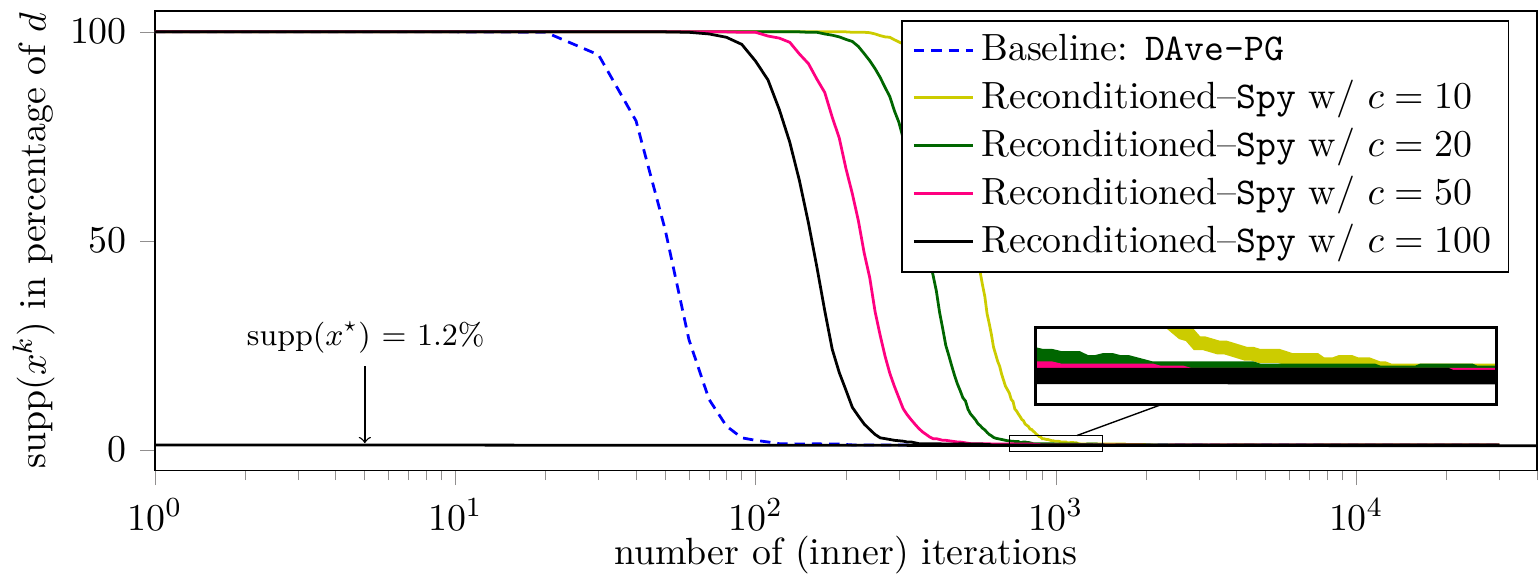} \\
\includegraphics[width=0.38\textwidth]{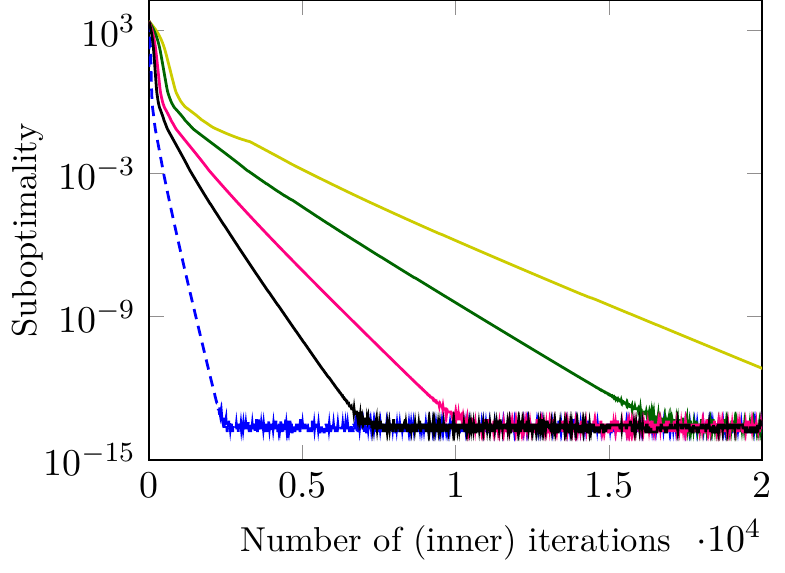}~~~~~~     \includegraphics[width=0.38\textwidth]{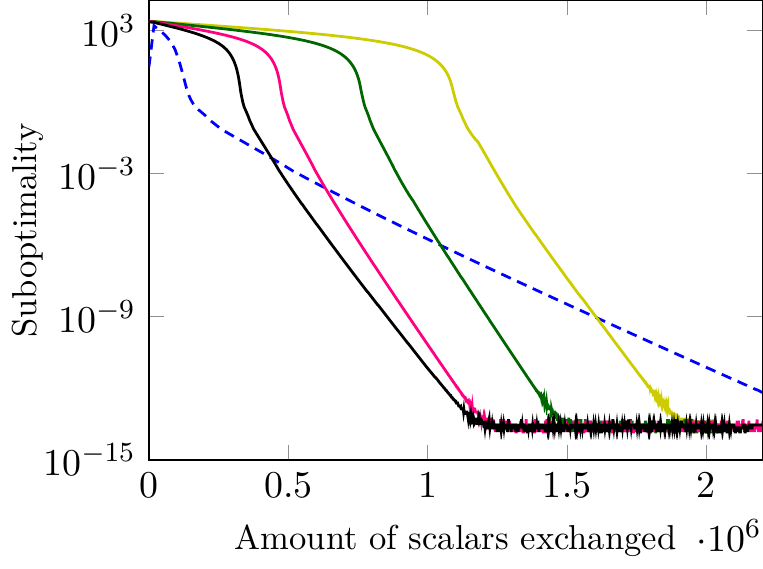}
    \caption{Comparison %Support of the iterates and performance %of  \dave and \recoalgo 
    on %a $500\times 1000$ 
    the lasso problem.\label{fig:lasso}}
\end{figure}
\begin{figure}[h!]
\includegraphics[width=0.7\textwidth]{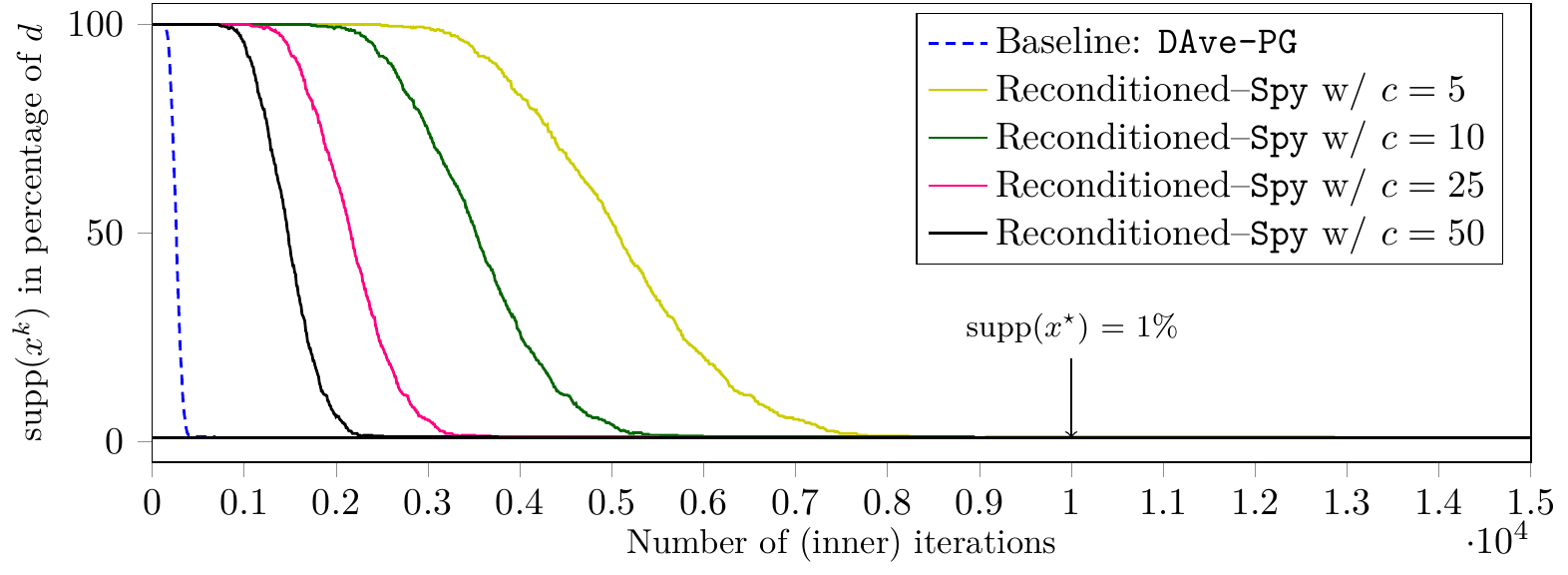} \\
\includegraphics[width=0.38\textwidth]{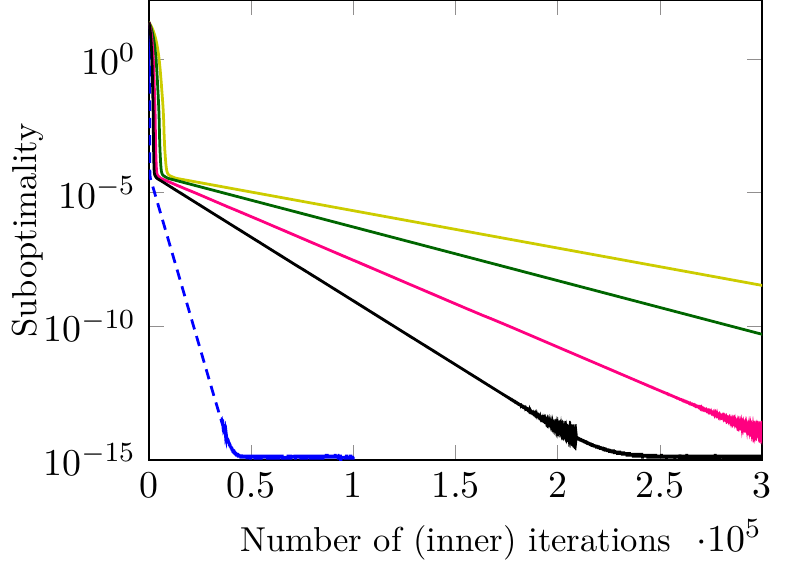}~~~~~~ \includegraphics[width=0.38\textwidth]{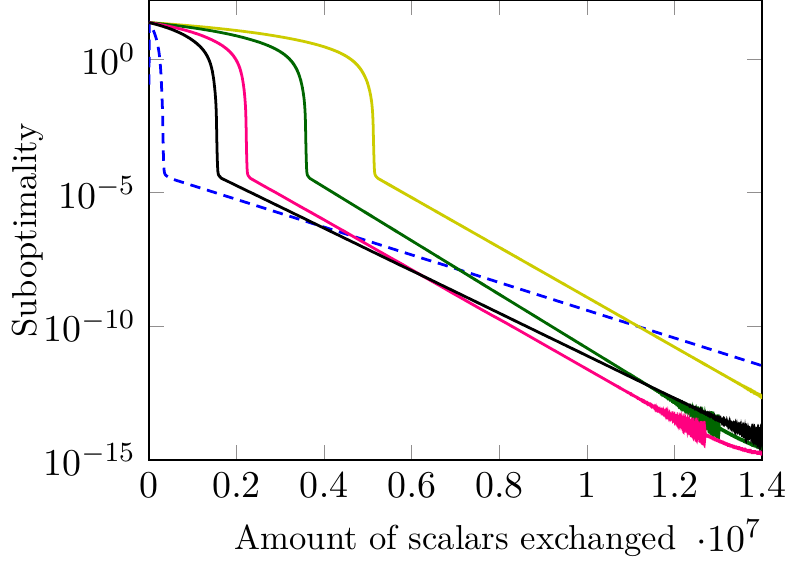}
    \caption{Comparison %Support of the iterates and performance %of \dave and \recoalgo 
    on the madelon logistic regression problem.\label{fig:madelon}}
\end{figure}

% \medskip
% \noindent
% \textbf{Final remarks.} 
%%%%%%%%%%%%%%%%%%%%%%%%%%%%%%%%%%%%%%%%%%%%%%%%%%%
\medskip
\noindent
\textbf{Empirical results.}
We make the following observations from experiments with both problems (reported in Figures\;\ref{fig:lasso} and\;\ref{fig:madelon} respectively).
When the support of iterates is far from the optimal one,
%identification takes places, 
sparsification is generally bad for convergence in terms of iterations (as shown by the slopes in the plots ``suboptimality vs iterations''), but even in terms of communications (see the beginning of the curves ``suboptimality vs exchanges'). On the other hand, when the iterates begin to be closer to the optimal support, adaptive sparsification becomes highly beneficial as illustrated by the final slopes of the plots ``suboptimality vs exchanges''. 

Since there is no guarantee that the currently identified support is the optimal one, it is impossible to restrict ourselves to a subset of the coordinates; here comes the need for our adaptively sparsified method, that keeps exploring dimensions, additionally to those in the current support. The quantity of 
randomly chosen dimensions, controlled by $c$ has an slight impact: we see on that %the figures %(Fig.\;\ref{fig:lasso}, \ref{fig:madelon} es) 
that (relatively) small and large values of $c$ (yellow and black curves) lead to slightly worse slopes on the convergence plots, compared to $c$ being in the range of 1 to 3 times the optimal support (green and pink curves) which is our recommendation both theoretically (see Theorem~\ref{th:com}) and in practice.

\begin{figure}[h!]
 \includegraphics[width=0.4\textwidth]{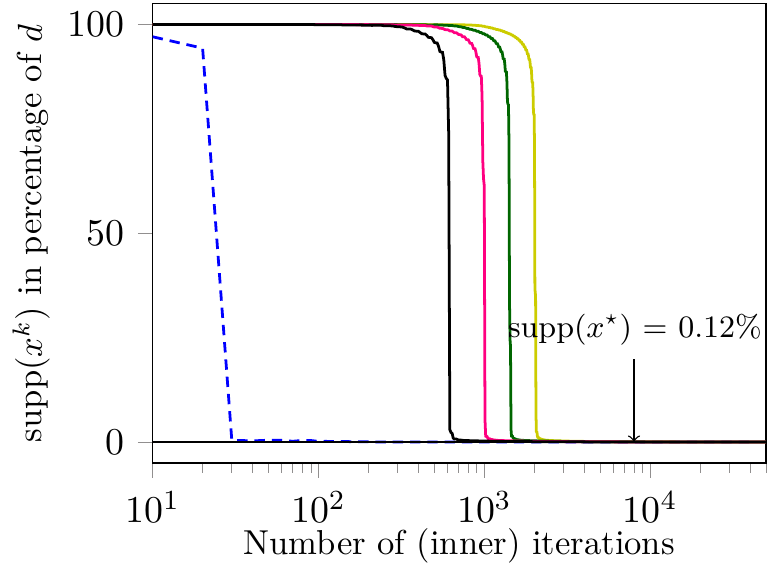}~~~~~~ \includegraphics[width=0.38\textwidth]{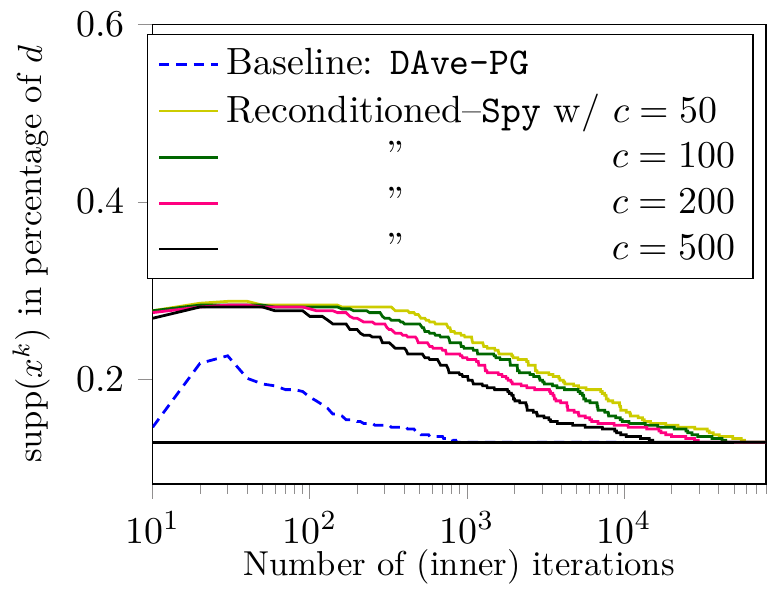} \\
\includegraphics[width=0.4\textwidth]{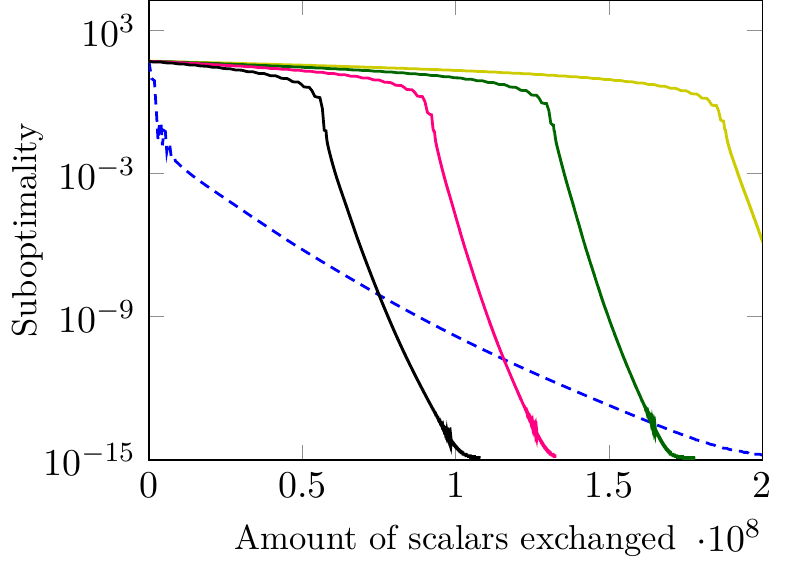}~~~~~~     \includegraphics[width=0.38\textwidth]{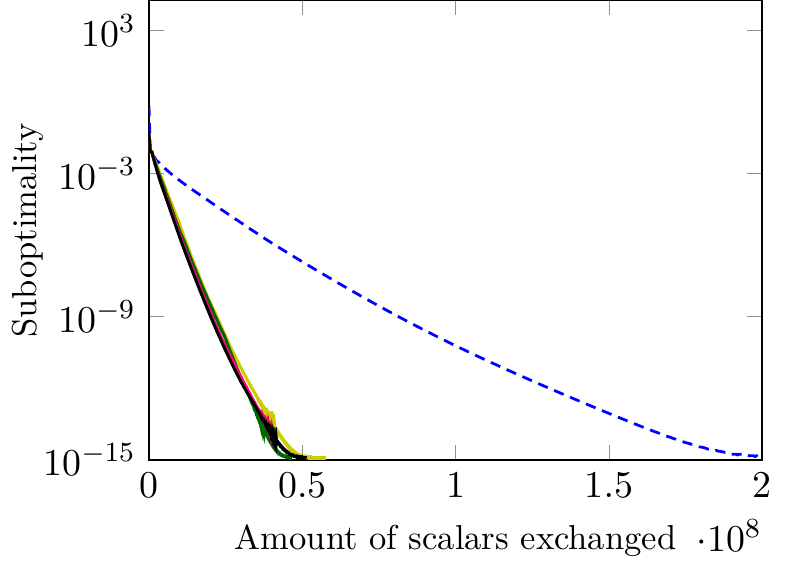}
    \caption{Comparison %Support of the iterates and performance %of \dave and \recoalgo 
    on the rcv1 %$47236 \times 20242$ 
    logistic regression problem. On the right-hand side plots, the algorithms are warm-started to an initial %using XXX epochs of \dave leading to a 
    suboptimality of $10^{-2}$ and density of $1\%$, reached % \dave initially exchanges around $10^6$ scalar to reach this accuracy, which is 
    within less than $5\%$ of the total number of exchanges to target precision.\label{fig:rcv1}}
\end{figure}

Finally, in order to mitigate the above-mentioned negative effects of sparsification in the first iterations for large problems, we propose to use a warmstart strategy: in the first iterations, we use a non-sparsified (eg. \dave) or a moderately sparsified method %, such as \dave in order 
to allow for a sharp initial functional decrease, leading to some partial identification; after this warmstart we switch to our sparsified method to fully benefit from identification. This strategy is illustrated in Figure\;\ref{fig:rcv1}, with warmstarted %\footnote{on this problem, \dave exchanges around $10^6$ scalar to reach this accuracy, which is less than $5\%$ of the total number of exchanges to reach the target precision.}
algorithms on the right-hand-side vs.\;the non-warmstarted ones on the left-hand-side.
We see a drastic improvement in terms of communication offered by the quick identification, for all versions of the sparsified method.

\bibliographystyle{siamplain}
\bibliography{references}

\renewcommand\appendixname{}
\appendix

\clearpage

\begin{center}
    \bfseries  \LARGE \textsc{Supplementary Materials}
\end{center}

\renewcommand{\thesection}{SM\arabic{section}}
\setcounter{figure}{0}
\renewcommand{\thefigure}{SM\arabic{figure}}

% ===========================================================
% ===========================================================
\section{Inefficiency of uniform sparsification illustrated}\label{apx:num}
% ===========================================================
% ===========================================================

In Section\;\ref{sec:adapt}, we discussed the impact of the random selection in our sparsification technique. The uniform sparsification is proved to have a degraded convergence rate in terms of epochs; we illustrate in  Figure\;\ref{fig:uniform} that moreover it shows a degraded empirical performance in terms of exchanges.

\begin{figure}[h!]
\includegraphics[width=0.38\textwidth]{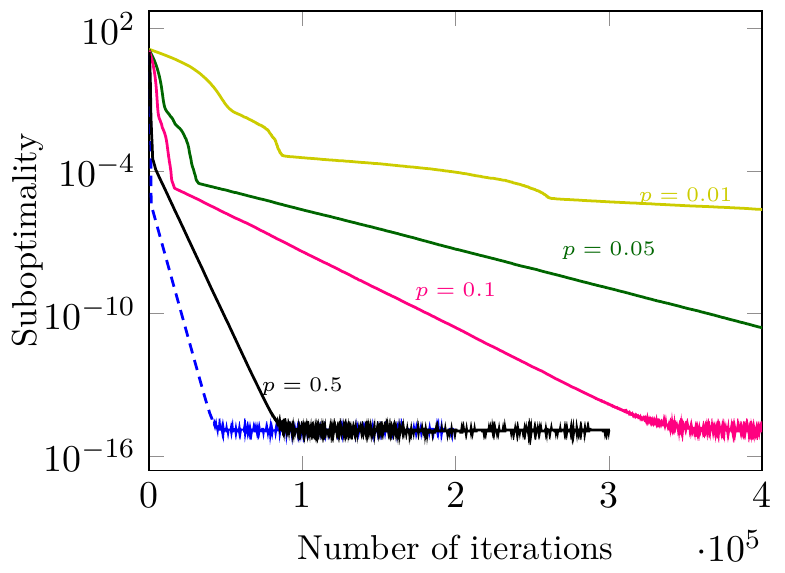}~~~~~~ \includegraphics[width=0.38\textwidth]{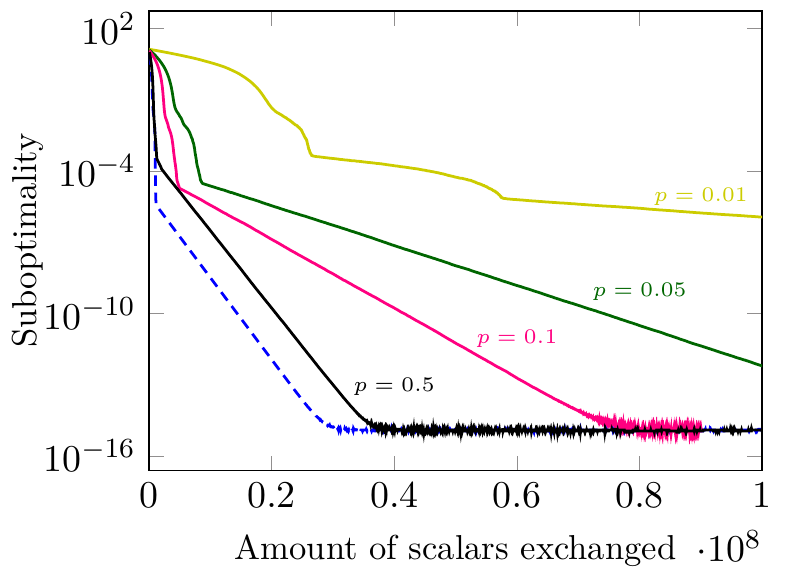}
    \caption{Illustration of the inefficiency of uniform sparsification on the madelon logistic regression problem of Section\;\ref{sec:num}.  We compare \dave (in dashed lines) and \salgo (in solid lines) with different values of the sparsification probability $p$ (indicated by the number next to the plots).  \label{fig:uniform}}
\end{figure}

% ===========================================================
% ===========================================================
\section{Proof of Theorem~\ref{lm:spy_diff}}\label{apx:proofsparse}
% ===========================================================
% ===========================================================
From the solution $x^\star$ of~\eqref{eq:pb} (unique from strong convexity), we define for each worker $i$, the local shift $ x_i^\star = x^\star - \gamma \nabla f_i(x^\star)$. From those, one can define $\barx^\star = \sum_{i=1}^M \alpha_i x_i^\star$. First-order optimality conditions 
\begin{align*}
    0\in \sum_i \alpha_i \nabla f_i(x^\star) + \partial r(x^\star)
\end{align*}
imply that 
\begin{align}\label{eq:xbar}
\barx^\star = \sum_{i=1}^M \alpha_i x_i^\star = x^\star - \gamma \sum_{i=1}^M \alpha_i \nabla f_i(x^\star) ~\in~ x^\star  + \gamma \partial r(x^\star)
\end{align}
which directly leads to $\rgprox(\barx^\star) = x^\star$ (see Chap.\;16 of\;\cite{bauschke2011convex}).

We can now lay down the proof using these above-defined variables. For a time $k$ and a worker $i$, we have that $x_i^k = x_i^{k-d_i^k}$ depends on i) $x^{k-D_i^k}$ which is $\mathcal{F}^{k-D_i^k}$-measurable; and ii) $\mathbf{S}^{k-D_i^k}$ which is i.i.d.\,. First, we are going to control the term $\|x_i^k - x_i^\star\|^2$.

Let us define $\|x\|_p^2 = \sum_{i=1}^d p_ix_{[i]}^2$ where  $(p_1,\dots,p_d)$ is the vector of probabilities of Assumption~\ref{hyp:algo}. The conditional  expectation can be developed as follows:
\begin{align*}
    \mathbb{E}[\|x_i^k - x_i^\star\|^2 | \mathcal{F}^{k-D_i^k} ]
    =~ &  \mathbb{E}[\|x_i^{k-d_i^k} - x_i^\star\|^2 | \mathcal{F}^{k-D_i^k} ] = \sum_{j=1}^d  \mathbb{E}[ (x_{i[j]}^{k-d_i^k} - x_{i[j]}^\star)^2 | \mathcal{F}^{k-D_i^k} ] \\ 
    =~ & \| x^{k-D_i^k} - \gamma \nabla f_i(x^{k-D_i^k})  - (  {x}^{\star} - \gamma \nabla f_i(x^\star) ) \|^2_p + \| x_{i}^{k-D_i^k} - {x}_{i}^{\star}\|^2_{1-p}.
\end{align*}
Let us now bound both terms of this sum above using $\pmax = \max_i p_i$ and $\pmin = \min_i p_i$.
\begin{align*}
   &  \| x^{k-D_i^k} - \gamma \nabla f_i(x^{k-D_i^k})  - (  {x}^{\star} - \gamma \nabla f_i(x^\star) ) \|^2_p     + \| x_{i}^{k-D_i^k} - {x}_{i}^{\star}\|^2_{1-p}\\
    & ~~~~ \leq  \pmax\| x^{k-D_i^k} - \gamma \nabla f_i(x^{k-D_i^k})  - (  {x}^{\star} - \gamma \nabla f_i(x^\star) ) \|^2  + (1-\pmin)\| x_{i}^{k-D_i^k} - {x}_{i}^{\star}\|^2.
\end{align*}
We now use the $\mu$-strong convexity and $L$-smoothness of $f_i$ to write (see e.g. \cite[Chap. 3.4.2]{bubeck2015convex}),
\begin{align*}
    \| x^{k-D_i^k} &- \gamma \nabla f_i(x^{k-D_i^k})  - (  {x}^{\star} - \gamma \nabla f_i(x^\star) ) \|^2\\
    & \leq \left( 1 - \frac{2\gamma \mu L}{\mu+L} \right) \left\| x^{k-D_i^k} - x^\star \right\|^2  - \gamma \left( \frac{2}{\mu + L} - \gamma\right)   \left\|\nabla f_i(x^{k-D_i^k}) - \nabla f_i(x^\star) \right\|^2 \\
        & \leq \left[ \left( 1 - \frac{2\gamma \mu L}{\mu+L} \right) - \mu^2 \gamma \left( \frac{2}{\mu + L} - \gamma\right)  \right] \left\| x^{k-D_i^k} - x^\star \right\|^2  \\
        &=     \left( 1 - \gamma \mu \right)^2 \left\| x^{k-D_i^k} - x^\star \right\|^2.
\end{align*}
Thus, for any $\gamma \in (0,2/(\mu+L)]$, 
\begin{align*}
    \mathbb{E}[\|x_i^k - x_i^\star\|^2 | \mathcal{F}^{k-D_i^k} ]&\leq
     \pmax \left(1-\gamma\mu\right)^2 \left\| x^{k-D_i^k} - x^\star \right\|^2 + (1-\pmin) \| x_{i}^{k-D_i^k} - {x}_{i}^{\star}\|^2 \\
    &\leq \pmax \left(1-\gamma\mu\right)^2 \left\| \barx^{k-D_i^k} - \barx^\star \right\|^2 + (1-\pmin) \| x_{i}^{k-D_i^k} - {x}_{i}^{\star}\|^2,
\end{align*}
where we used that $\| x^{k-D_i^k} - x^\star \|^2 = \| \rgprox (\barx^{k-D_i^k}) - \rgprox(\barx^\star) \|^2 \leq \| \barx^{k-D_i^k} - \barx^\star \|^2 $ by definition and non-expansiveness of the proximity operator of $r$.

Taking full expectation on both sides, we get 
\begin{align*}
    \mathbb{E}\|x_i^k - x_i^\star\|^2 &\leq \pmax \left(1-\gamma\mu\right)^2 \mathbb{E}\left\| \barx^{k-D_i^k} - \barx^\star \right\|^2 + (1-\pmin) \mathbb{E}\| x_{i}^{k-D_i^k} - {x}_{i}^{\star}\|^2.
\end{align*}
Then,  using that $\barx^{k-D_i^k} - \barx^\star  =  \sum_{i=1}^M \alpha_i ( {x}_i^{k-D_i^k} - \barx_i^\star)$ and the convexity of $\|\cdot\|^2$, we get
\begin{align*}
    \mathbb{E}\|x_i^k - x_i^\star\|^2 
   \leq ~ &\pmax \left(1-\gamma\mu\right)^2 \sum_{j=1}^M \alpha_j \mathbb{E}\left\| x_j^{k-D_i^k} - x_j^\star \right\|^2 
 + (1-\pmin)\mathbb{E} \| x_{i}^{k-D_i^k} - {x}_{i}^{\star}\|^2 \\
   \leq ~ &\pmax \left(1-\gamma\mu\right)^2 \max_{j=1,\dots,M} \mathbb{E}\left\| x_j^{k-D_i^k} - x_j^\star \right\|^2 
 + (1-\pmin) \max_{j=1,\dots,M}  \mathbb{E}\| x_{j}^{k-D_i^k} - {x}_{j}^{\star}\|^2 \\
   \leq ~ & \left( \pmax\left(1-\gamma\mu\right)^2  + 1 - \pmin \right) \max_{j=1,\dots,M} \mathbb{E}\left\| x_j^{k-D_i^k} - x_j^\star \right\|^2.
\end{align*}
Let $c_k = \max_{i=1,\dots,M} \mathbb{E}\left\| x_j^{k} - x_j^\star \right\|^2$ and $\beta = \left( \pmax\left(1-\gamma\mu\right)^2  + 1 - \pmin \right)$ (note that the assumptions imply that $\beta\leq 1$), then the above result implies that
$$
    c_{k} \leq \beta \max_{j=1,\dots,M} c_{k-D_j^k}
$$
and using the definition of the sequence $(k_m)$, we get 
\begin{align*}
    c_{k_m} &\leq \beta ~ \max_{j} c_{k_m-D_j^{k_m}} \leq  \beta \max_{\ell\in[k_{m-1},k_m)} c_{\ell} \\
  c_{k_m+1} &\leq  \beta ~ \max( c_{k_m} , \max_{\ell\in[k_{m-1},k_m)} c_{\ell}) \leq  \beta \max_{\ell\in[k_{m-1},k_m)} c_{\ell}.
\end{align*}
Thus for all $k \geq k_m$, $ c_k \leq  \beta ~ \max_{\ell\in[k_{m-1},k_m)} c_{\ell}$. This implies that the sequence $\widetilde{c}_{m}$ defined by $\widetilde{c}_{m}= \max_{\ell\in[k_{m},k_{m+1})} c_{\ell}$ has an exponential bound:
\begin{equation*}
    \widetilde{c}_{m} \leq   \beta ~ \widetilde{c}_{m-1} \leq \beta^m ~ \widetilde{c}_{0}\leq \beta^m ~ \max_{i=1,\dots,M} \|x_i^0 - x_i^\star\|^2.
\end{equation*}
Finally, we to use once again the non-expansivity of the proximity operator of $r$ and the definitions to get that for all $k\in [k_m,k_{m+1}),$
\begin{align}\label{eq:spy_rate}
 \mathbb{E}\|x^k - x^\star\|^2 ~\leq~ \mathbb{E}\|\barx^k - \barx^\star\|^2 
    ~\leq~ \sum_{i=1}^M \alpha_i \mathbb{E}\|{x}_i^k - {x}_i^\star\|^2
    ~\leq~ c_k ~\leq~ \beta^m  \max_{i=1,\dots,M} \|x_i^0 - x_i^\star\|^2,
\end{align}
which concludes the proof.

% ===========================================================
% ===========================================================
\section{Discussion on non-i.i.d. selections}\label{sec:noniid} 
% ===========================================================
% ===========================================================

Imagine that one wants to build an algorithm similar to \salgo but where the update probabilities are equal to $1$ over an adaptively chosen set $\mathsf{A}^k$ (depending e.g. on the support of $x^k$) and $\pp<1$ for the others. As mentioned in the main text, this poses a problem since the selection is not i.i.d. anymore due to the adaptive selection. 

This drawback can actually be fixed by ``slowing down'' the updates in the adaptive set $\mathsf{A}^k$ by a factor $\pp$. This gives Algorithm~\ref{alg:fixspy}. A proof very similar to the one of Supplement~\ref{apx:proofsparse} enables to prove a $(1-\pp\gamma\mu(2-\gamma\mu))$ linear rate (the same as for uniform selection with probability $\pp$). Unfortunately, the artificial slowdown of the iterates is harmful in practice and causes this algorithm to be rather slow, which is why we chose not to consider this path.

\begin{algorithm}
\caption{Adaptive-\textsc{\salgo} on $((\alpha_i),(f_i), r )$ with stopping criterion $\mathsf{C}$}
\label{alg:fixspy}

\tcbset{width=0.48\columnwidth,before=,after=\hfill, colframe=black,colback=white, fonttitle=\bfseries, coltitle=white, colbacktitle=black, boxrule=0.2mm,arc=0mm, left = 2pt}

\begin{tcolorbox}[title=\textsc{Coordinator} \vphantom{\texttt{Worker  i}}]
Initialize $\bar x^0$\\
%Send $\overline x$ to each machine\\
\While{test $\mathsf{C}$ not verified}{
    \hspace*{0ex}\\
     {\color{blue!70!black}Receive\,$[\Delta^k]_{\mathbf{S}^{k-D_{i^k}^k}}$\,from\,agent\,$i^k$}\\
     $\barx^k \leftarrow \barx^{k-1} + \alpha_i[\Delta^k]_{\mathbf{S}^{k-D_{i^k}^k}} $\\
     $ x^k \leftarrow \rgprox(\bar x^k )$\\
    Draw sparsity mask $\mathbf{S}^k$ from $\mathsf{A}^k,\pp$\\[0.5ex]
    {\color{red!80!yellow} Send $x^k, \mathbf{S}^k, \mathsf{A}^k$ to agent $i^k$}\\
    $k\leftarrow k+1$
}
%Interrupt all workers\\
%\textbf{Output} $x^k$
\end{tcolorbox}
\tcbset{width=0.48\columnwidth, colframe=black!50!black, coltitle=white, colbacktitle=black}
\begin{tcolorbox}[title=\textsc{Worker } $i$]
Initialize $x_i = x_i^+=0$\\
\While{not interrupted}{
    \hspace*{0.1ex}\\
     {\color{red!80!yellow}Receive  $x$, $\mathbf{S}$, and  $\mathsf{A}$ from  the coordinator}\\
    $[x^+_i]_{\mathbf{S}}   \leftarrow [  x - \gamma \nabla f_i( x)]_{\mathbf{S}} $\\[0.6ex]
        $[x^+_i]_{\mathsf{A}}   \leftarrow \pp [x^+_i]_{\mathsf{A}}  + (1-\pp)[x_i]_{\mathsf{A}}  $\\[0.6ex]
    $\Delta \leftarrow x_i^+ - x_i $ \\[0.7ex]
    {\color{blue!70!black}Send $[\Delta]_{\mathbf{S}}$ to the coordinator}\\[0.7ex]
    $[x_i]_{\mathbf{S}} \leftarrow [x^+_i]_{\mathbf{S}} $%\\[0.7ex]
}
%\vspace*{0.55cm}
\end{tcolorbox}
\vspace*{-0cm}
\end{algorithm}

% ===========================================================
% ===========================================================
\section{Proofs of convergence for reconditioned algorithms} \label{apx:reco}
% ===========================================================
% ===========================================================

% ===========================================================
\subsection{Basic lemmas for proofs of reconditioned algorithms}
% ===========================================================

We state here two simple lemmas that we have not found as such in the literature and that are required in the proofs of Section~\ref{sec:recoalgo}. They use the fact that the unique minimum $x^\star$ of a strongly convex function $F$ is a fixed point of $\prox_{F/{\rho }}$ for any $\rho>0$; see\;\cite[Prop.~12.28]{bauschke2011convex}).

\begin{lemma}
    \label{lem:opstrong}
    Let $F: \mathbb{R}^d \to \mathbb{R}\cup\{+\infty\}$ be $\mu$-strongly convex lsc and $\rho>0$.
    Then, %$F$ has a unique minimizer $x^\star$ and 
    for any $x\in\mathbb{R}^d$,
    \begin{align*}
        \| \prox_{F/{\rho}}(x) - \prox_{F/{\rho}}(x^\star) \|^2 \leq \frac{{\rho}}{2\mu + {\rho}} \| x - x^\star \|^2 - \frac{{\rho}}{2\mu + {\rho}} \| \prox_{F/{\rho}}(x) - x \|^2.
    \end{align*}
\end{lemma}
\begin{proof}
    The proof simply consists in developing norms as follows:
    \begin{align*}
        \| &\prox_{F/{\rho}}(x) - x \|^2 \\
        &= \| \prox_{F/{\rho}}(x) - \prox_{F/{\rho}}(x^\star) + x^\star - x \|^2 \\
        &= \| \prox_{F/{\rho}}(x) - \prox_{F/{\rho}}(x^\star) \|^2 + \| x-  x^\star  \|^2 - 2 \langle   \prox_{F/{\rho}}(x) - \prox_{F/{\rho}}(x^\star) ; x-  x^\star \rangle  \\
        &\leq  \| \prox_{F/{\rho}}(x) - \prox_{F/{\rho}}(x^\star) \|^2 + \| x-  x^\star  \|^2 - 2(1+\mu/{\rho}) \| \prox_{F/{\rho}}(x) - \prox_{F/{\rho}}(x^\star) \|^2 \\
        &=  \| x-  x^\star  \|^2 - ( 1 + 2\mu/{\rho}) \| \prox_{F/{\rho}}(x) - \prox_{F/{\rho}}(x^\star) \|^2;
    \end{align*}
    then a reordering concludes the proof. The inequality uses the the fact that the resolvent of the $\mu/{\rho}$ strongly monotone operator $\partial F/{\rho}$ is $(1+\mu/{\rho})$-cocoercive; see\;\cite{bauschke2011convex}, particularly Proposition~23.11.
\end{proof}

\begin{lemma}
    \label{lem:basereco}
    Let $F: \mathbb{R}^d \to \mathbb{R}\cup\{+\infty\}$ be $\mu$-strongly convex lsc, $\rho>0$. Then, for any $x,x' \in\mathbb{R}^d$ such that 
    \begin{align*}
      \EE \left[  \left\|x' - \prox_{F/{\rho }}(x) \right\|^2 | x \right] \leq  \nu  \left\|x - \prox_{F/{\rho }}(x) \right\|^2 
    \end{align*}
    for some $\nu>0$, we have for any $\varepsilon\in(0,1)$
    \begin{align}
       \label{eq:basestrong} 
        \EE \left[   \left\|x' - x^\star  \right\|^2 | x \right] &\leq \left( 1 + \varepsilon \right)  \frac{{\rho }}{2\mu + {\rho }}  \left\|x - x^\star \right\|^2 \\
   \nonumber     &\hspace*{0.8cm} -\left[ (1 + \varepsilon) \frac{{\rho }}{2\mu + {\rho }}  - \left( 1 + \frac{1}{\varepsilon} \right)  \nu  \right]  \|x -  \prox_{F/{\rho }}(x)   \|^2.
     \end{align}
\end{lemma}

\begin{proof}
     Using Young's inequality, we get that for any $\varepsilon\in(0,1)$,
\begin{align*}
    \EE & \left[ \left\|x' - x^\star  \right\|^2  | x \right] \\
    &\leq \left( 1 + \frac{1}{\varepsilon} \right)  \EE \left[ \left\|x' -  \prox_{F/{\rho }}(x)   \right\|^2  | x \right]   +   \left( 1 + \varepsilon \right)\left\| \prox_{F/{\rho }}(x) - \prox_{F/{\rho }}(x^\star ) \right\|^2 \\
    &\leq \left( 1 + \frac{1}{\varepsilon} \right)  \nu  \|x -  \prox_{F/{\rho }}(x)   \|^2 +   \left( 1 + \varepsilon \right)  \frac{{\rho }}{2\mu + {\rho }} \left\|x - x^\star \right\|^2 \\
    &\hspace*{0.8cm} -  \left( 1 + \varepsilon \right)  \frac{{\rho }}{2\mu + {\rho }}  \left\| x - \prox_{F/{\rho }}(x)  \right\|^2 \\
     &= \left( 1 + \varepsilon \right)  \frac{{\rho }}{2\mu + {\rho }}  \left\|x - x^\star \right\|^2  -\left[ (1 + \varepsilon) \frac{{\rho }}{2\mu + {\rho }}  - \left( 1 + \frac{1}{\varepsilon} \right)  \nu  \right]  \|x -  \prox_{F/{\rho }}(x)   \|^2.
 \end{align*}
where we used Lemma~\ref{lem:opstrong}.
\end{proof}

% ===========================================================
\subsection{Proof of Theorem~\ref{th:reco}}
% ===========================================================

Theorem~\ref{th:reco} gives very similar results for the three inner stopping criteria. However, the proof techniques are quite different for each criteria.

% ===========================================================
\vspace*{1ex}
\paragraph*{Proof for criterion  $\mathsf{C}_1$ \emph{(epoch budget)}}

We start by noticing that, at outer loop $\ell$, \salgo solves the reconditioned problem \eqref{eq:approxalgo} over which it has a contraction factor of $(1-\alpha_\ell)$ with $1-\alpha_\ell  := 1-\alpha+\pp - \pp_\ell$; see~\eqref{eq:reco_contra} and~\eqref{eq:ratebefore}. This means that \salgo initialized with $x_\ell$ verifies after $m$ epochs with the maximal\footnote{The initialization means that $x^0 = x_\ell$ and $x_i^0 =  x_\ell - \gamma \nabla h_{i,\ell}(x_\ell)$. The use of the maximal stepsize is required here to guarantee the optimal rate over the reconditioned problem, see \eqref{eq:ratebefore}.} stepsize
    \begin{align*}
             \mathbb{E}   \left\| x_\ell^{k_m} - x_\ell^\star \right\|^2 \le \left( 1-\alpha_\ell  \right)^{m} \max_i\left\|x_i^0-x_{i,\ell}^\star\right\|^2 \le \left( 1-\alpha_\ell  \right)^{m} \left\|x_\ell -x_{\ell}^\star\right\|^2
    \end{align*}
    where $x_\ell^\star$ is the unique solution of \eqref{eq:reco}, and $x_{i,\ell}^\star = x_\ell^\star - \gamma \nabla h_{i,\ell}(x^\star_\ell)$ are its local shifts.
    
    We now apply Lemma~\ref{lem:basereco} with the following input: $x' = x_{\ell+1} = x_\ell^{k_{\mathsf{M}_\ell}}$; $x = x_\ell$; $F = F$ (which is $\mu$ strongly convex); $x^\star =x^\star$ (minimizer of $F$); and $\nu = \left( 1-\alpha_\ell  \right)^{\mathsf{M}_\ell}$. We get for $\varepsilon = \frac{1}{\ell^{1+\delta}} \in (0,1)$
        \begin{align*}
        \EE \left[   \left\|x_{\ell+1} - x^\star  \right\|^2 | x_\ell \right] &\leq \left( 1 + \frac{1}{\ell^{1+\delta}} \right)  \frac{{\rho }}{2\mu + {\rho }}  \left\| x_\ell - x^\star \right\|^2 \\
   \nonumber     &\hspace*{-1.2cm} - \underbrace{\left[ \left(1 + \frac{1}{\ell^{1+\delta}}\right) \frac{{\rho }}{2\mu + {\rho }}  - \left( 1 + \ell^{1+\delta} \right)  (1-\alpha_\ell)^{\mathsf{M}_\ell}  \right]}_{:= b_\ell}  \|x_\ell -  \prox_{F/{\rho }}(x_\ell)   \|^2.
     \end{align*}
     Choosing $\mathsf{M}_\ell$  as per $\mathsf{C}_1$ guarantees that
     $$
     b_\ell \geq \delta \left( 1 + \frac{1}{\ell^{1+\delta}} \right)  \frac{{\rho }}{2\mu + {\rho }} \geq    \frac{{\delta  \rho }}{2\mu + {\rho }} .
     $$ 
     Thus, for any $\mu\geq 0$, %($0$ included), 
     we have
     \begin{align}\label{eq:C1base}
        \EE \left[   \left\|x_{\ell+1} - x^\star  \right\|^2 | x_\ell \right] &\leq \left( 1 + \frac{1}{\ell^{1+\delta}} \right) \frac{{ \rho }}{2\mu + {\rho }}  \left\| x_\ell - x^\star \right\|^2 -  \frac{{\delta  \rho }}{2\mu + {\rho }} \|x_\ell -  \prox_{F/{\rho }}(x_\ell)   \|^2.
     \end{align}
     
     \vspace*{1ex}
     \noindent\underline{Convergence.} For any $\mu\geq 0$, \eqref{eq:C1base} tells us that  
          \begin{align}
          \label{eq:RS}
        \EE \left[   \left\|x_{\ell+1} - x^\star  \right\|^2 | x_\ell \right] &\leq \left( 1 + \frac{1}{\ell^{1+\delta}} \right)  \left\| x_\ell - x^\star \right\|^2 -  \delta \|x_\ell -  \prox_{F/{\rho }}(x_\ell)   \|^2 .
     \end{align}
     By Robbins-Siegmund theorem \cite[Th.~1]{robbins1971convergence} (see also \cite{iutzeler2013asynchronous,bianchi2015coordinate,combettes2015stochastic} for applications to optimization), we have that i) $(\left\|x_{\ell} - x^\star  \right\|^2)$ converges almost surely to a random variable with finite support; and ii) $ \sum_{\ell=1}^\infty \|x_\ell -  \prox_{F/{\rho }}(x_\ell)   \|^2 < \infty$. This means that we can extract a subsequence $(x_{\ell^n})$ that converges almost surely to some $y$ which is necessarily a minimizer from ii). Using~\eqref{eq:RS} again with $x^\star = y$, we see that $(x_{\ell})$ converges to $x^\star$ almost surely.
     
     \vspace*{1ex}
     \noindent\underline{Rate.} Now, if $\mu>0$, we get by dropping the last term in~\eqref{eq:C1base} and successively taking expectations that  
      \begin{align}
        \EE \left[   \left\|x_{\ell+1} - x^\star  \right\|^2 \right] &\leq \left( {\ell} + 1  \right)^{1+\delta}  \left( \frac{{\rho }}{2\mu + {\rho }} \right)^\ell  \left\| x_1 - x^\star \right\|^2 
        = \tilde{\mathcal{O}} \left( \left( 1 - \frac{{\mu }}{\mu + {\rho }/2} \right)^\ell \right).
     \end{align}

% ===========================================================
\vspace*{1ex}
\paragraph*{Proof for criterion $\mathsf{C}_2$ \emph{(absolute accuracy)}}

We apply Lemma\;\ref{lem:basereco} with the following input:\;$x'\!= x_{\ell+1}$; $x = x_\ell$; $F = F$ (which is $\mu$ strongly convex); $x^\star =x^\star$ (minimizer of $F$); 
%\rho = \rho$; 
and $ \nu = (1-\delta)\rho/((2\mu+\rho)\ell^{1+\delta})$ (noting that the condition on $x_{\ell+1}$ is almost sure). We get for $\varepsilon = \frac{1}{\ell^{1+\delta}} \in (0,1)$
        \begin{align*}
  %\label{eq:C2base} 
  \left\|x_{\ell+1} - x^\star  \right\|^2 &\leq \left( 1 + \frac{1}{\ell^{1+\delta}} \right)  \frac{{\rho }}{2\mu + {\rho }}  \left\| x_\ell - x^\star \right\|^2 \\
   \nonumber     &\hspace*{-1.2cm} - \underbrace{\left[ \left(1 + \frac{1}{\ell^{1+\delta}}\right) \frac{{\rho }}{2\mu + {\rho }}  - \left( 1 + \ell^{1+\delta} \right)  \frac{1}{\ell^{1+\delta}}  \frac{{(1-\delta)\rho }}{2\mu + {\rho }} \right]}_{\geq 0}  \|x_\ell -  \prox_{F/{\rho }}(x_\ell)   \|^2 \\
 \nonumber  &\leq  \left( {\ell} + 1  \right)^{1+\delta}  \left( \frac{{\rho }}{2\mu + {\rho }} \right)^\ell  \left\| x_1 - x^\star \right\|^2 .
     \end{align*}
     This directly gives the rate of convergence when $\mu>0$. 
    When $\mu=0$, the inequality can be simplified to
    \begin{align}
  \nonumber \left\|x_{\ell+1} - x^\star  \right\|^2 &\leq \left( 1 + \frac{1}{\ell^{1+\delta}} \right)    \left\| x_\ell - x^\star \right\|^2 \\
   \nonumber     &\hspace*{-1.2cm} -  \left[ \left(1 + \frac{1}{\ell^{1+\delta}}\right)   - \left( 1 + \ell^{1+\delta} \right)  \frac{1}{\ell^{1+\delta}}  (1-\delta) \right]  \|x_\ell -  \prox_{F/{\rho }}(x_\ell)   \|^2 \\
\nonumber  &\leq \left( 1 + \frac{1}{\ell^{1+\delta}} \right)    \left\| x_\ell - x^\star \right\|^2 - \delta  \|x_\ell -  \prox_{F/{\rho }}(x_\ell)   \|^2.
     \end{align}
     In this case, the same arguments as for criterion $\mathsf{C}_1$ enable to get almost sure convergence.

% ===========================================================
\vspace*{2ex}
\paragraph*{Proof for criterion $\mathsf{C}_3$ \emph{(relative accuracy)}}

Denoting $\beta := \sqrt{\rho/(2\mu+\rho)} \in (0,1]$, the stopping criterion $\mathsf{C}_3$ writes
\begin{align}\label{eq:baseC3}
    \|x_{\ell+1} -  \prox_{F/\rho}(x_\ell)\| \leq   \varepsilon_\ell \|x_{\ell+1} -  x_\ell \|  \qquad\text{ with }    \varepsilon_\ell =  \frac{\beta}{2\ell^{1+\delta}} .
\end{align}
% This is the usual inexactness in proximal algorithms \cite{rockafellar1976monotone}.
%\vspace*{1ex}
\noindent\underline{Convergence.} The condition\;\eqref{eq:baseC3} matches condition\;(B) of\;\cite[Th.~2]{rockafellar1976monotone}. We also have clearly $\sum_\ell \varepsilon_\ell < +\infty$ and the regularity assumption of the operator is verified, by our extra assumption and   \cite[Prop.~7]{rockafellar1976monotone}. Thus \cite[Th.~2]{rockafellar1976monotone} directly gives us that $(x_\ell)$ converges to a minimizer of $F$, that we denote by $x^\star$.

\vspace*{1ex}
\noindent\underline{Rate.} When $F$ is $\mu$-strongly convex, we can furthermore develop:
\begin{align*}
    \left\| x_{\ell+1} - x^\star \right\| &\leq \left\| x_{\ell+1} - \prox_{F/\rho}(x_\ell) \right\| + \left\|\prox_{F/\rho}(x_\ell) - x^\star \right\| \\
    &\leq \varepsilon_\ell \left\| x_{\ell+1} - x_\ell \right\| + \beta \left\|x_\ell - x^\star \right\| \\
    &\leq  \varepsilon_\ell \left\| x_{\ell+1} - x^\star \right\|  + \varepsilon_\ell \left\| x_{\ell} - x^\star \right\| + \beta \left\|x_\ell - x^\star \right\|
\end{align*}
where the first inequality used both condition $\mathsf{C}_3$ and Lemma~\ref{lem:opstrong}. This implies that 
\begin{align}
\label{eq:C3better}
    \left\| x_{\ell+1} - x^\star \right\|^2 &\leq \left( \frac{\varepsilon_\ell + \beta}{1 - \varepsilon_\ell} \right)^2  \left\|x_\ell - x^\star \right\|^2 .
\end{align}
Finally, denoting $d_\ell := \ell^{1+\delta}/(\ell^{1+\delta}-1) > 1 $, we have\footnote{Note that $
   \varepsilon_\ell^2 = \frac{\rho}{4(2\mu+\rho) \ell^{2+2\delta}} \leq \left( \frac{\beta}{2} \right)^2 \frac{(d_\ell-1)^2}{d_\ell^2} \leq \beta^2 \frac{(d_\ell-1)^2}{(1+\beta d_\ell)^2} $.}
\begin{align*}
    \varepsilon_\ell \leq \beta \frac{(d_\ell-1)}{(1+\beta d_\ell)} ~~\Longrightarrow~~ \frac{\varepsilon_\ell + \beta}{1 - \varepsilon_\ell} \leq d_\ell \beta \leq \ell^{1+\delta}/((\ell-1)^{1+\delta}) \beta.
\end{align*}
This yields
\begin{align*}
    \left\| x_{\ell+1} - x^\star \right\|^2 &\leq \left( \frac{\ell}{\ell-1} \right)^{2+2\delta} \frac{{\rho }}{2\mu + {\rho }} \left\|x_\ell - x^\star \right\|^2 
\end{align*}
which gives the result.

% ===========================================================
% ===========================================================
\section{Acceleration \emph{à la} Catalyst}\label{apx:cata} 
% ===========================================================
% ===========================================================
Catalyst is a popular \emph{accelerated} inexact proximal point meta-algorithm for solving machine learning objectives \cite{lin2015universal,lin2017catalyst}.
It consists in adding an acceleration step for each outer iteration to reach a faster rate. The parameters of Catalyst are i) the choice of the inertial sequence $(\beta_\ell)$; and ii) the choice of the stopping criteria that comes in the same flavor as for \recoalgo. Applied to our context, this gives Algorithm\;\ref{algo:cata} where the main difference with \recoalgo is the final line with the acceleration. 

Under the parameters given in \cite{lin2017catalyst}, the fast outer rate is stated in the following result which shows in the square root compared to Theorem~\ref{th:reco} for the direct proximal reconditioning.

\begin{theorem}
\label{thm:cata}
Let the functions $(f_i)$ be $\mu$-strongly convex ($\mu\geq0$) and $L$-smooth. Let $r$ be convex lsc. 
Then, \cataalgo on $((\alpha_i),(f_i),r)$ with stopping criterion $\mathsf{C}'_2$ or $\mathsf{C}'_3$  converges in terms of suboptimality at rate $1/\ell^2$.

Furthermore, if $\mu>0$, then 
      \begin{align}
      \left\|x_{\ell+1} - x^\star  \right\|^2  ={\mathcal{O}} \left( \left(   1 - \sqrt{ \frac{{\mu }}{4\mu + 4{\rho }} }  \right)^\ell \right) .
        \end{align}
\end{theorem}

\begin{proof}
   The result directly follows from \cite{lin2017catalyst}, in particular Propositions 5,6,8, and 9.
\end{proof}

Note that $\mathsf{C}_1$ is not supported in this result, since it is not covered by theory in \cite{lin2017catalyst}; however, a criterion close to $\mathsf{C}_1$ is used in the computational experiments of \cite{lin2017catalyst}.

\begin{algorithm}[h!]
\caption{\label{algo:cata}\cataalgo on $((\alpha_i),(f_i),r)$}
\begin{flushleft}
Initialize $x_1$, $c>0$, $\delta\in(0,1)$, as well as $\rho$ and $\gamma$ as in \eqref{eq:set}.

% \begin{align}
% \label{eq:setCata}
%  \text{Set } \rho = \frac{\kappa L-\mu}{1-\kappa} \text{ and } \gamma \in \left( 0, \frac{2}{\mu+L+2\rho} \right] \text{ with } \kappa = \frac{1-\sqrt{\pp -\alpha}}{1+\sqrt{\pp-\alpha}}; \pp = \frac{c}{d} \text{ and } \alpha = \frac{c}{2d}.
% \end{align}
\While{the desired accuracy is not achieved}{
Observe the support of $y_\ell$, compute $\pvec_\ell$ as in \eqref{eq:proba}.
% \begin{align}
% p_{j,\ell} = \left\{ \begin{array}{cr}
%  \pp_\ell := \min\left(\frac{c}{|\nullC(y_\ell)|};1\right) & \text{ if } [x_\ell]_j = 0 \\
%   1  &  \text{ if } [x_\ell]_j \neq  0 
% \end{array}   \right. ~~~~~ \text{for all $j\in\{1,\dots,d\}$}.
% \end{align}
Compute an approximate solution of the reconditioned problem
$$%\begin{align} 
x_{\ell+1} \approx \prox_{F/\rho} (y_\ell)
% \argmin_{x\in\mathbb{R}^d}  ~ \left\{ H_\ell(x) = \sum_{i=1}^M  \alpha_i \underbrace{\left( f_i(x) + \frac{\rho}{2} \| x - y_\ell \|_2^2 \right)}_{h_{i,\ell}(x)}  +  r(x)  \right\}
$$%\end{align}
with \salgo on $((\alpha_i),(h_{i,\ell}), r  ~ ; ~  \pvec_\ell)$  with stopping criterion:\\%[0.1cm]
\begin{itemize}
    % \item[] $\mathsf{C}_1$: \emph{(epoch budget)} Run \salgo, initialized with $x_\ell$, with the maximal stepsize for
    % $$  \displaystyle \mathsf{M}_\ell  \text{ epochs.}$$
    \item[] $\mathsf{C}'_2$ (absolute accuracy):  Run \salgo until it finds $x_{\ell+1}$ such that  
    $$ H_\ell(x_{\ell+1}) - \min_x H_\ell(x)  \leq 
    \left\{ 
\begin{array}{ll}
\left( 1 - \sqrt{\frac{\mu}{4(\mu+\rho)}} \right)^\ell  \frac{2 ( F(y_{1}) - \min_x  F(x) ) }{9} &  \text{if } \mu>0 \\
\frac{1}{\ell^{4+\delta}}  \frac{2 ( F(y_{1}) - \min_x  F(x) ) }{9}  & \text{if } \mu=0 
\end{array}
\right.   .$$ 
    \item[or] $\mathsf{C}'_3$ (relative accuracy):  Run \salgo until it finds $x_{\ell+1}$ such that  
$$ H_\ell(x_{\ell+1}) - \min_x H_\ell(x)  \leq
\left\{ 
\begin{array}{ll}
  \frac{\sqrt{\mu}}{2\sqrt{\mu+\rho}- \sqrt{\mu}} \frac{  \rho  \| x_{\ell+1} - y_\ell \|^2 }{2}    &  \text{if } \mu>0 \\
 \frac{1}{\ell^2} \frac{  \rho  \| x_{\ell+1} - y_\ell \|^2  }{2 }    & \text{if } \mu=0
\end{array}
\right. .$$ 
    Compute the next point with Nesterov’s extrapolation step
    $$ y_{\ell+1} = x_{\ell+1} + \beta_\ell (x_{\ell+1}-x_\ell). $$
    \vspace*{-2ex}
\end{itemize}
}
\end{flushleft}
\end{algorithm}

As for the other two criteria, $\mathsf{C}'_2$  and $\mathsf{C}'_3$ imply rules similar to  $\mathsf{C}_2$  and $\mathsf{C}_3$ since
    $$ \|x_{\ell+1} -  \prox_{F/\rho}(x_\ell)\|^2 \leq \frac{H_\ell(x_{\ell+1}) - H_\ell\big(\prox_{F/\rho}(x_\ell)\big)}{\mu+\rho}
    = \frac{H_\ell(x_{\ell+1}) - \min_x  H_\ell(x)}{\mu+\rho}. $$
    by using the $(\mu+\rho)$-strong convexity of the inner problem. We note that $\mathsf{C}'_2$  is then much more stringent that $\mathsf{C}_2$. For $\mathsf{C}'_3$, the conditions are similar to $\mathsf{C}_3$  in the non-strongly convex case, however they appear to be better in the strongly convex case. This is an artefact due to our use of one criterion both strongly convex and non-strongly convex cases. Indeed, in the strongly convex case, one can see from \eqref{eq:C3better} that a fixed choice of $\varepsilon_\ell$ is possible, leading to a similar strategy and rate.

% ===========================================================
% ===========================================================
\section{Proofs related to identification}
% ===========================================================
% ===========================================================

% ===========================================================
\subsection{Proof of Theorem~\ref{th:ident}}\label{apx:ident}
% ===========================================================

We notice that \dave, as many proximal algorithms (see e.g.\;\cite[Sec.\;4]{iutzeler2020SVAA}), has the following convergence property
\begin{equation}\label{eq:convprox}
 \bar x^k \to \bar{x}^\star \text{ and thus } x^k = \prox_{\gamma r}( \bar x^k ) \to {x}^\star = \prox_{\gamma r}(  \bar{x}^\star )
\end{equation}
from Theorem\;\ref{th:davepg}. This property is the key to get the identification result, as we will develop at the end of the proof. We start by proving that \recoalgo (as well as \cataalgo) shares a similar convergence property.

Both \recoalgo and \cataalgo are of the form 
\begin{align*}
    x_{\ell+1} &\approx \prox_{F/\rho}(y_\ell)
    \qquad\text{and}\qquad
    y_{\ell+1} = x_{\ell+1} + \beta_\ell(x_{\ell+1} - x_\ell) 
\end{align*}
(with $\beta_\ell \equiv 0$ for \recoalgo and $\beta_\ell \in (0,1)$ for \cataalgo); and verify for all $\ell,k$
\begin{align*}
    \mathbb{E} \|x_{\ell} - x^\star \| &\leq C (1-\rho)^\ell
    \qquad\text{and}\qquad
    \mathbb{E} \|\barx^k_{\ell} - \barx_\ell^\star \| \leq C'  \|y_{\ell} - x_\ell^\star \| 
\end{align*}
for some $C,C'>0$ and $\rho\in(0,1)$, where 
$$
\barx_\ell^\star :=  {x}_\ell^\star - \gamma \sum_{i=1}^M \alpha_i \nabla h_{i,\ell}({x}_\ell^\star).
$$
As in Supplement~\ref{apx:proofsparse}, we also consider $\barx^\star = x^\star - \gamma \sum_{i=1}^M \alpha_i \nabla f_i(x^\star)$ given by \eqref{eq:xbar}. 
Then, we have
\begin{align*}
    \|\barx_\ell^\star - \barx^\star\| &=    \|{x}_\ell^\star - \gamma \sum_{i=1}^M \alpha_i \nabla h_{i,\ell}({x}_\ell^\star) - {x}^\star + \gamma \sum_{i=1}^M \alpha_i \nabla f_i({x}^\star) \| \\ 
    &=  \|{x}_\ell^\star - \gamma \sum_{i=1}^M \alpha_i \nabla f_{i}({x}_\ell^\star) - \gamma \rho ({x}_\ell^\star - {y}_\ell) - {x}^\star + \gamma \sum_{i=1}^M \alpha_i \nabla f_i({x}^\star) \| \\ 
    &\leq  \|{x}_\ell^\star - {x}^\star\| + \gamma \sum_{i=1}^M \alpha_i  \|\nabla f_i({x}_\ell^\star) - \nabla f_i({x}^\star)\| + \gamma \rho \|{x}_\ell^\star - {y}_\ell\| \\
    &\leq  \|{x}_\ell^\star - {x}^\star\| + \gamma \sum_{i=1}^M \alpha_i L \|{x}_\ell^\star - {x}^\star\| + \gamma \rho \|{x}_\ell^\star - {x}_\ell\| \\
    &\leq  D \|{x}_\ell^\star - {x}^\star\|  + D' \|y_\ell - {x}^\star_\ell\|
\end{align*}
for some $D,D'>0$. For any $\ell,k$, we then have
\begin{align*}
    \EE \|\overline x_{ {\ell}}^k &-  \overline {x}^\star\|^2  \\
    & \leq 2\EE\left[\|\overline x_{ {\ell}}^k -  \overline x^{\star}_{ {\ell}}\|^2 + \|\overline x^{\star}_{ {\ell}} -  \barx^\star\|^2\right] \\
    &\leq 2C' \EE \|y_{\ell} - x_{ {\ell}}^\star\|^2 + 2D  \EE\|x^{\star}_{ {\ell}} -  {x^\star}\|^2 + 2D'  \EE\|y_\ell - {x}^\star_\ell\|^2\\
    &\leq 4C' \EE \|y_{\ell} - x^\star\|^2  + (4C'+2D) \EE \|x_{\ell}^\star - x^\star\|^2 + 2D'  \EE\|y_\ell - {x}^\star_\ell\|^2  \\
    &\leq (8C'+2D+2D') \EE \|y_{\ell} - x^\star\|^2 \\
    &\leq (8C'+2D+2D') \EE \|(1+ {\beta}_\ell)({ {x_{\ell}}} -  {x^\star}) -  {\beta_\ell}({ {x_{\ell-1}}} -  {x^\star})\|^2 + 2D'  \EE\|x^{\star}_{ {\ell}} -  {x_\ell}\|^2\\
    &{\leq} 2(8C'+2D+2D')(1+ {\beta}_\ell)^2\EE \|{ {x_{\ell}}} -  {x^\star})\|^2  +2(8C'+2D+2D')(1+ {\beta}_\ell)^2 \EE\|{ {x_{\ell-1}}} -  {x^\star})\|^2\\
    &{\leq} 2(8C'+2D+2D')(1+ {\beta}_\ell)^2 C (1-\rho)^\ell +2(8C'+2D+2D')(1+ {\beta}_\ell)^2 C (1-\rho)^{\ell-1} \\
    &\leq 16(8C'+2D+2D') C  (1-\rho)^{\ell-1}.
    \end{align*}
Hence, by Markov's inequality and Borel-Cantelli's lemma, $\overline x_{\ell}^k \to  \overline {x}^\star$ almost surely. As a direct result, we get for our two random algorithms, the same convergence as \eqref{eq:convprox} for \dave,
$$
 \barx_{\ell}^k \to_{\ell \to \infty} \barx^\star \text{ and thus } x_{\ell}^k = \prox_{\gamma r}( \barx_{\ell}^k ) \to_{\ell \to \infty} {x}^\star = \prox_{\gamma r}(  \barx^\star ) ~~ \text{ with probability }1.
$$

This convergence implies identification of optimal support (see e.g.\;the recent survey\;\cite[Cor.\;1]{iutzeler2020SVAA}). Recalling Assumption~\ref{hyp:ident}\emph{i}, there exists $\varepsilon>0$ such that 
$$ \supp(x^\star) = \supp \left( \prox_{\gamma r} \left( x^\star -\gamma \sum_{i=1}^M  \alpha_i  \nabla f_i(x^\star) + \mathbf{e}   \right)   \right) = \supp \left( \prox_{\gamma r} \left( \barx^\star  + \mathbf{e}   \right)   \right)$$
for all $ \mathbf{e} \in \mathcal{B}(0,\gamma \varepsilon) $. Hence, since $\barx_{\ell}^k \to_{\ell \to \infty} \bar{x}^\star $ almost surely, $\barx_{\ell}^k $ will belong to the ball of radius $\gamma\varepsilon$ centered on $\barx^\star$ in a finite number of outer iterations, say $\Lambda<\infty$, with probability one. The equation above then directly implies that for all $\ell\geq\Lambda$, $\supp(x^\star) = \supp(\prox_{\gamma r}( \barx_{\ell}^k ) ) = \supp(x_{\ell}^k)$.  

Finally, it suffices to notice that $x_{\ell+1} = x_\ell^k$ for some $k$ to conclude the proof.

% ===========================================================
\subsection{Proof of Theorem~\ref{th:imprate}}\label{apx:imprate}
% ===========================================================

From Theorem~\ref{th:ident} we know that identification takes place i.e.\;that we have $\nullC(x^k_\ell) = \nullC(x^\star) := n^\star$ for all $k,\ell$ ($\ell\geq\Lambda$).
In this case, 
\begin{align*}
    [x^k_\ell]_{n^\star} = 0 \quad\text{ and }\quad [x^k_\ell]_{\overline{n^\star}} = x_\ell^k
\end{align*}
where we denote by $\overline{n^\star}$ the complementary of $n^\star$. In addition, if $x=\rgprox(\barx)$ is such that $\nullC(x)=n^\star$, then $x=[x]_{\overline{n^\star}}=\rgprox([\barx]_{\overline{n^\star}})$  by separability of $r$. Then, recalling that all coordinates in $\supp(x_\ell^k) = \overline{n^\star}$ are updated with the choice of~\eqref{eq:proba}, we have for all $k,\ell\geq \Lambda$
\begin{align*}
  \nonumber   [x^k_\ell]_{\overline{n^\star}}  
  &=\left[\prox_{\gamma r}([\barx^k_\ell]_{\overline{n^\star}})\right]_{\overline{n^\star}}  
  = \left[\prox_{\gamma r}\left(\sum_{i=1}^M \alpha_i [x^{k-D_i^k}]_{\overline{n^\star}} - \gamma \sum_{i=1}^M \alpha_i [\nabla f_i(x^{k-D_i^k})]_{\overline{n^\star}} \right)\right]_{\overline{n^\star}} \\
     &=  \left[\prox_{\gamma r}\left(\sum_{i=1}^M \alpha_i x^{k-D_i^k} - \gamma \sum_{i=1}^M \alpha_i [\nabla f_i(x^{k-D_i^k})]_{\overline{n^\star}} \right)\right]_{\overline{n^\star}}.
\end{align*}
This exactly coincides with a non-sparsified update on the restriction of $f_i$ to the subspace of vectors with null coordinates in ${\overline{n^\star}}$. More specifically, let $S^\star = \{ x\in \mathbb{R}^d : \nullC(x) = n^\star \}$ be the subspace of vectors with null coordinates in ${\overline{n^\star}}$, and
$f_{i|{\overline{n^\star}}}$ be the restriction of $f_i$ to $S^\star$.
Then the above iteration coincides with non-sparsified update on
$((\alpha_i),(f_{i|{\overline{n^\star}}}),r)$. In other words, after identification, \salgo is no longer random and has the same iterates as \dave on $S^\star$ (while in $S^{\star\perp}$, the algorithm has converged to $0$). Theorem\;\ref{th:davepg} therefore guarantees that \salgo benefits from a $(1-\gamma (\mu+\rho))^2$ rate in terms of epochs (since $(\mu+\rho)$ is the modulus of strong convexity of each $f_{i|{\overline{n^\star}}}$).

% ===========================================================
\subsection{Proof of Theorem~\ref{th:com}}\label{apx:com}
% ===========================================================

Following the choice of~\eqref{eq:proba},
\begin{align*}
    \mathbf{M}^{\mathsf{C}} = \tilde{O}\left( \frac{1}{\gamma(\mu+\rho)} \right)
\end{align*}
for both $\mathsf{C}_2$ and $\mathsf{C}_3$ from Theorem~\ref{th:imprate} (see also \cite[Lemma~11]{lin2017catalyst}). Finally, 
\begin{align*}
    \mathbf{L}(\varepsilon) = {O}\left( \frac{\mu + \rho/2}{\mu} \log\left( \frac{1}{\varepsilon} \right) \right) 
\end{align*}
from Theorem~\ref{th:reco}.
Put together, this gives the following complexity:
\begin{align*}
     \mathbf{C}(\varepsilon) = \tilde{O}\left( \frac{\mu + \rho/2}{\gamma \mu (\mu+\rho)} (2s^\star + c) \log\left( \frac{1}{\varepsilon} \right) \right).
\end{align*}
    Since $c\approx s^\star \ll d$, we have by definition (in \eqref{eq:rho}) $\rho\gg L\geq\mu$ (we will use  $\mu+L\leq 2(\sqrt{2}-1)\rho$ in the $\lesssim$ step of the upcoming equation). Taking the maximal $\gamma = 2/(\mu+L+2\rho)$, we have 
    \begin{align*}
        \frac{\mu + \rho/2}{\gamma \mu (\mu+\rho)} (2s^\star + c) &\leq \frac{\mu + L + 2\rho}{2  \mu} (2s^\star + c)\\
        &\lesssim_{c\ll d} \frac{\sqrt{2}(\kappa_{\eqref{eq:reco}} L-\mu)(2s^\star + c) }{ \mu (1-\kappa_{\eqref{eq:reco}})} \\
        &\leq 2 \frac{(L-\mu)\sqrt{d}(2s^\star + c) }{  \mu \sqrt{ c}} \\
        &= 2 \frac{L-\mu}{\mu} \sqrt{d c } \left(1+2\frac{s^\star}{c} \right)\\
        &=  2 \frac{L-\mu}{\mu} \sqrt{d s^\star } \left(\sqrt{\frac{c}{s^\star}}+2\sqrt{\frac{s^\star}{c}} \right) \leq 4 \frac{L-\mu}{\mu} \sqrt{d s^\star } ~ \max \left(\sqrt{\frac{c}{s^\star}};\sqrt{\frac{s^\star}{c}} \right).
    \end{align*}
    
% ===========================================================
% ===========================================================
\section{Numerical illustration of the stopping tests}\label{apx:C3}
% ===========================================================
% ===========================================================
In Algorithm\;\ref{algo:reco}, the stopping criteria $\mathsf{C}_1$ with prescribed number of inner loops is simple and natural. In contrast, enforcing criteria $\mathsf{C}_3$ needs a \emph{full} gradient evaluation (see e.g.\;\cite{rockafellar1976monotone}), which requires some partial synchronization and %  The issue would then be that this computation  
%Though it is possible to implement it, it would 
then breaks the asynchronous nature of the method. 

In our numerical experiments, we use a simplification of the stopping test $\mathsf{C}_1$ with only one epoch, following the principle of ``1 pass over the data = 1 restart'' used in Catalyst \cite{lin2017catalyst}. We illustrate here that this simple stopping rule gives similar empirical convergence as $\mathsf{C}_3$ without its computational issues.

\begin{figure}
    \includegraphics[width=0.7\textwidth]{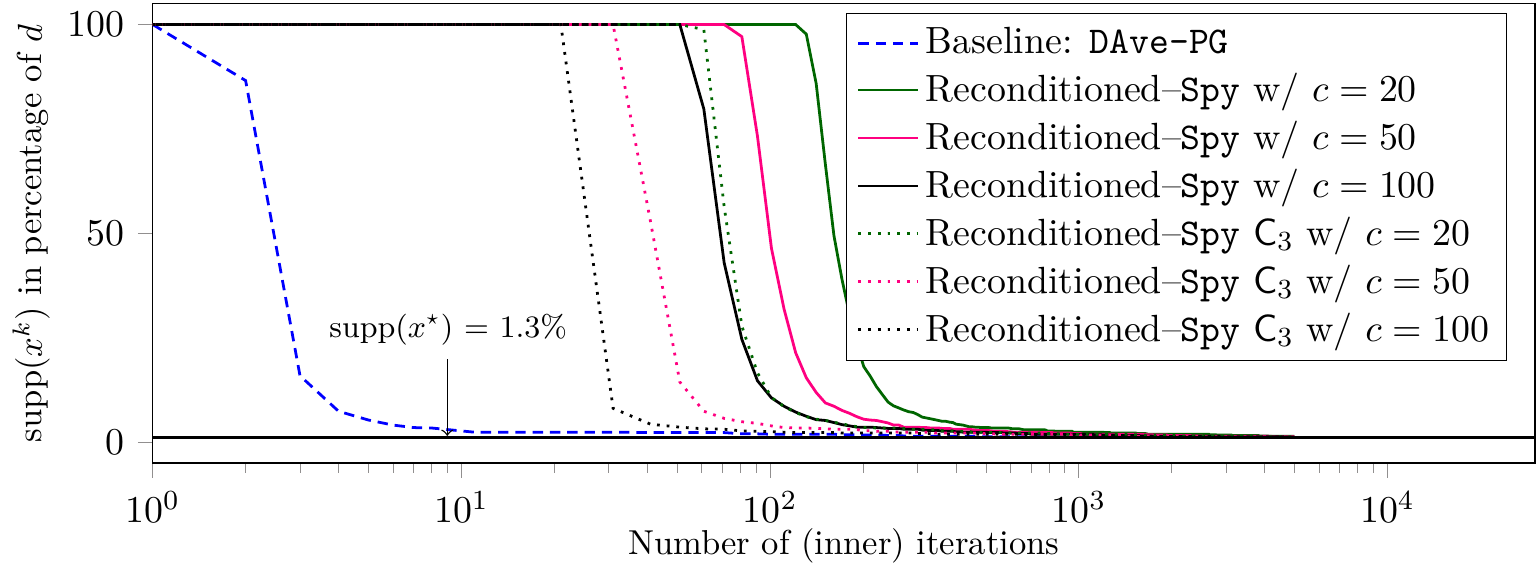} \\
\includegraphics[width=0.38\textwidth]{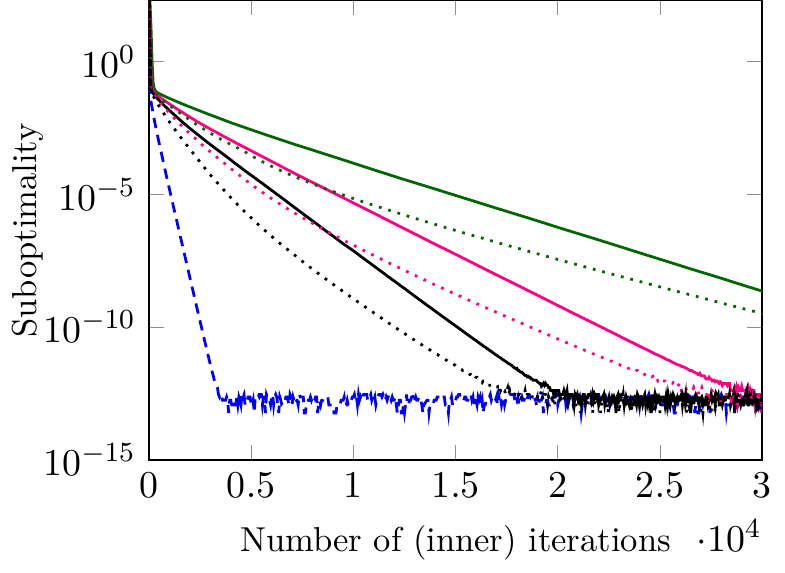}     \includegraphics[width=0.38\textwidth]{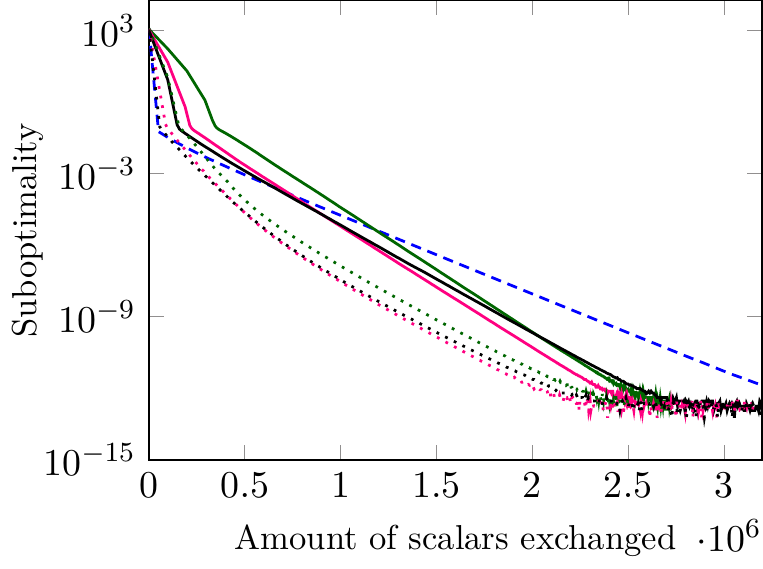}
    \caption{Comparison of $\mathsf{C}_1$ with one epoch and $\mathsf{C}_3$ on a LASSO problem. }
\end{figure}

\end{document}